\documentclass{amsart}
\usepackage{latexsym,amsxtra,amscd,ifthen}
\usepackage{amsfonts}
\usepackage{verbatim}
\usepackage{amsmath}
\usepackage{amsthm}
\usepackage{amssymb}

\numberwithin{equation}{section}

\theoremstyle{plain}
\newtheorem{theorem}{Theorem}[section]
\newtheorem{lemma}[theorem]{Lemma}
\newtheorem{proposition}[theorem]{Proposition}
\newtheorem{corollary}[theorem]{Corollary}
\newtheorem{conjecture}[theorem]{Conjecture}

\theoremstyle{definition}
\newtheorem{definition}[theorem]{Definition}
\newtheorem{example}[theorem]{Example}

\newtheorem{remark}[theorem]{Remark}
\newtheorem{question}[theorem]{Question}

\makeatletter              
\let\c@equation\c@theorem  
\makeatother

\DeclareMathOperator{\hdet}{hdet}

\DeclareMathOperator{\Ext}{Ext} \DeclareMathOperator{\Tor}{Tor}

 \DeclareMathOperator{\Gr}{Gr}

\DeclareMathOperator{\Aut}{Aut}

\DeclareMathOperator{\Kdim}{Kdim}
\DeclareMathOperator{\Cdim}{Cdim}

\DeclareMathOperator{\GKdim}{GKdim}

\begin{document}

\title[Noncommutative Complete intersections]
{Noncommutative Complete intersections}

\author{E. Kirkman, J. Kuzmanovich and J.J. Zhang}

\address{Kirkman: Department of Mathematics,
P. O. Box 7388, Wake Forest University, Winston-Salem, NC 27109}

\email{kirkman@wfu.edu}

\address{Kuzmanovich: Department of Mathematics,
P. O. Box 7388, Wake Forest University, Winston-Salem, NC 27109}

\email{kuz@wfu.edu}

\address{zhang: Department of Mathematics, Box 354350,
University of Washington, Seattle, Washington 98195, USA}

\email{zhang@math.washington.edu}

\begin{abstract}
Several generalizations of a commutative ring that is a graded
complete intersection are proposed for a noncommutative graded
$k$-algebra; these notions are justified by examples from
noncommutative invariant theory.
\end{abstract}

\subjclass[2010]{16E65, 16W50}




\keywords{Artin-Schelter regular algebra, complete intersection,
group action, Gorenstein, Hilbert series, quasi-bireflection}


\maketitle


\setcounter{section}{-1}
\section{Introduction}
\label{xxsec0}

Bass has noted that Gorenstein rings are ubiquitous \cite{Ba}. Since
the class of Gorenstein rings contains a wide variety of rings, it
has proven useful to consider a tractable class of Gorenstein rings, and
complete intersections fill that role for commutative rings.
Similarly Artin-Schelter Gorenstein algebras, which are
noncommutative generalizations of commutative Gorenstein rings,
include a diverse collection of algebras, and finding a class of
Artin-Schelter Gorenstein algebras that generalizes the class of
commutative complete intersections is an open problem in
noncommutative algebra.  In this paper we use our work in
noncommutative invariant theory to propose several notions of a
noncommutative graded complete intersection.  Moreover, the
existence of noncommutative analogues of commutative complete
intersection invariant subalgebras broadens our continuing project
of establishing an invariant theory for finite groups acting on
Artin-Schelter regular algebras that is parallel to classical
invariant theory  (see \cite{KKZ1}--\cite{KKZ5}).

When a finite group acts linearly on a commutative polynomial ring,
the invariant subring is rarely a regular ring (the group must be a
reflection group \cite{ShT}), but Gorenstein rings of invariants are
easily produced.  For example, Watanabe's Theorem (\cite{W1} or
\cite[Theorem 4.6.2]{Be}) states that the invariant subring of
$k[x_1, \cdots, x_n]$ under the natural action of a finite subgroup
of $SL_n(k)$ is always Gorenstein, where $k$ is a base field. In
previous work we have shown there is a rich invariant theory for
finite group (and even Hopf) actions on Artin-Schelter regular
[Definition \ref{xxdef3.1}] (or AS regular, for short) algebras; for
example there is a noncommutative version of Watanabe's Theorem
\cite[Theorem 3.3]{JZ}, providing conditions when the invariant
subring is AS Gorenstein.

Cassidy and Vancliff defined a factor ring $S/I$ of $S =
k_{(q_{ij})}[x_1, \dots, x_n]$, a skew polynomial ring, to be a
complete intersection if $I$ is generated by a regular sequence of
length $n$ in $S$ (hence $S/I$ is a finite dimensional algebra)
\cite[Definition 3.7]{CV}, and in \cite{Va} Vancliff considered
extending this definition to graded skew Clifford algebras. A
few examples of noncommutative (or quantum) complete intersections
have been constructed and studied along the line of factoring out
a regular sequence of elements \cite{BE, BO, Op}. Further,
different kinds of generalizations of a commutative complete
intersection have been proposed during the last
fifteen years \cite{Go, EG, BG, GHS}. 
Recent work on noncommutative (or twisted)  matrix factorizations 
\cite{CCKM}, derived representation schemes \cite{BFR}, 
noncommutative versions of support varieties and finite 
generation of the cohomology ring of a Hopf algebra 
(ideas similar to \cite{BWi, NWi}), as well as 
noncommutative crepant resolutions of 
commutative schemes \cite{Da}, advocate for  a better understanding 
of noncommutative complete intersections. A satisfactory
definition  will have  positive impact on several research areas. 

In the commutative graded case a connected graded algebra $A$ is
called a {\it complete intersection} if one of the following four
equivalent conditions holds [Lemma \ref{xxlem1.7}]
\begin{enumerate}
\item[(cci$^\prime$)]
$A\cong k[x_1, \dots, x_d]/(\Omega_1, \cdots, \Omega_n)$, where
$\{\Omega_1, \dots, \Omega_n\}$ is a regular sequence of homogeneous
elements  in $k[x_1, \dots, x_d]$ with $\deg x_i>0$.
\item[(cci)]
$A\cong C/(\Omega_1, \cdots, \Omega_n)$, where $C$ is a noetherian
AS regular algebra and  $\{\Omega_1, \dots, \Omega_n\}$ is a
regular sequence of normalizing homogeneous elements in $C$.
\item[(gci)]
The $\Ext$-algebra $E(A):=\bigoplus_{n=0}^\infty \Ext^n_A(k,k)$ of
$A$ has finite Gelfand-Kirillov dimension.
\item[(nci)]
The $\Ext$-algebra $E(A)$ is noetherian.
\end{enumerate}
We call such an algebra $A$ a {\it commutative complete intersection}.
In the noncommutative case, unfortunately, these four conditions are
not all equivalent, nor does (gci) or (nci) force $A$ to be
Gorenstein (Example \ref{xxex6.3}), making it unclear which property 
to use as the proper
generalization of a commutative complete intersection. A direct
generalization to the noncommutative case is condition (cci) which
involves considering regular sequences in {\it any} AS regular
algebra (in the commutative case the only AS regular algebras are
the polynomial algebras). Several researchers, including those whom
we have mentioned earlier, have taken an approach that uses regular
sequences. Though the condition (cci) seems to be a good definition
of a noncommutative complete intersection, there are very few tools
available to work with condition (cci) except for explicit
construction and computation, and it is not easy to show condition
(cci) fails since one needs to consider regular sequences in {\it
any} AS regular algebra.

We consider both conditions (gci) and (nci) as possible
definitions of a noncommutative complete intersection. One advantage
of this approach is that it covers a large class of examples
coming from noncommutative algebraic geometry 
and noncommutative ring theory. For example, let $R$
be any noetherian Koszul algebra of finite Gelfand-Kirillov dimension
(or GK-dimension, for short), then the Koszul dual $A:=R^!$ satisfies
both (gci) and (nci), since the $\Ext$-algebra $E(A)$ of $A$ is
 ${A}^{!} = R^{!!}= R$, which is noetherian of finite GK-dimension.
We provide some information about the relationship between
these different notions of noncommutative complete intersection
in the following theorem.

\begin{theorem}
\label{xxthm0.1}{\rm [Theorem \ref{xxthm1.12}(a)]}  Let $A$ be a
connected graded noncommutative algebra.
  If $A$ satisfies (cci),
then it satisfies (gci).
\end{theorem}

Example \ref{xxex6.3} shows that even both (gci) and (nci) do not 
imply (cci), and Example \ref{xxex6.2} shows that (gci) does not imply (nci).

We call an 
algebra $A$ {\it cyclotomic} if its Hilbert series is a quotient of integral 
polynomials, all of whose roots are roots of unity.  If $A$ satisfies 
(cci) (including the case that $A$ is a commutative complete 
intersection) then it is necessary (but not sufficient) for $A$ to be 
cyclotomic.    In the noncommutative case, under reasonable hypotheses, 
properties (gci) and (nci) also satisfy this necessary condition.

\begin{theorem}\label{xxthm0.2} {\rm [Theorem \ref{xxthm1.12}(b, c)]} 
If $A$ satisfies (gci) or (nci), and if the Hilbert series of $A$ is a rational 
function $p(t)/q(t)$ for some coprime integral polynomials $p(t), q(t) 
\in \mathbb{Z}[t]$ with $p(0) = q(0) =1$, then $A$ is cyclotomic.
\end{theorem}

In Section 2 we use the cyclotomic condition to show that certain 
Veronese subrings are
AS Gorenstein algebras that are not complete intersections in terms of any of our generalizations.

Many interesting examples arise from the noncommutative invariant theory
of AS regular algebras under finite group actions, and the following question
was one motivation for our work. Let $\Aut(A)$
denote the group of graded algebra automorphisms of $A$.

\begin{question}
\label{xxque0.3} Let $A$ be a (noncommutative) noetherian connected
graded AS regular algebra and $G$ be a finite subgroup of $\Aut(A)$.
Under what conditions is the invariant subring $A^G$ a complete
intersection?
\end{question}

When $A$ is a commutative polynomial ring over $\mathbb{C}$,
Question \ref{xxque0.3} was answered by Gordeev \cite{G2} and
Nakajima \cite{N3,N4} independently. A very important tool in the
classification of groups $G$, such that $k[x_1, \cdots, x_n]^G$ is a
complete intersection, is a result of Kac and Watanabe \cite{KW} and
Gordeev \cite{G1} independently; they prove that if $G$ is a finite
subgroup of $GL_n(k)$ and $k[x_1, \cdots, x_n]^G$ is a complete
intersection, then $G$ is generated by bireflections (i.e. elements
$g \in G$ such that rank($g-I) = 2$). This leads us to the next
natural question: what is a bireflection in the noncommutative
setting? We seek notions of complete intersection and bireflection
that lead to a generalization of the result of Kac-Watanabe-Gordeev.

In Section 3 we give the basic definitions in our noncommutative invariant theory. 
In the commutative case, we have \cite[p. 506]{St3}
$${\text{regular}} \quad \; \Longrightarrow \quad {\text{hypersurface}}
\quad \; \Longrightarrow \quad {\text{complete intersection}}$$
$$\qquad\quad
\Longrightarrow \quad {\text{Gorenstein}} \qquad \Longrightarrow
\quad {\text{Cohen-Macaulay}}.$$
In contrast to the commutative case, neither (gci) nor (nci) implies
Gorenstein in the noncommutative case [Example \ref{xxex6.3}].
On the other hand, when we work with invariant subrings
$A^G$, we have a satisfactory situation.  An algebra $A$ is 
called {\it cyclotomic Gorenstein} if (i) $A$ is
AS Gorenstein and (ii) $A$ is cyclotomic.

\begin{theorem}\label{xxthm0.4}{\rm [Theorem \ref{xxthm3.4}]}
Assume that $k$ is of characteristic zero.  Let $R$ be a connected 
graded noetherian Auslander regular algebra, and $G$ be a finite 
subgroup of $\Aut(R)$.  If $A = R^G$ satisfies any of (cci), (gci) or 
(nci) then $R^G$ is cyclotomic Gorenstein. 
\end{theorem}
\noindent
It follows from Theorem \ref{xxthm0.4} that when $A^G$ is not 
cyclotomic and AS Gorenstein we know that
it does not satisfy any of our conditions for generalizations of
complete intersections.

Our generalization of a reflection of a symmetric algebra to the
notion of a quasi-reflection of a noncommutative AS regular algebra
\cite[Definition 2.2]{KKZ1}, suggests that a reasonable definition
of a noncommutative bireflection is following: a graded algebra
automorphism $g$ of a noetherian AS regular algebra $A$ of
GK-dimension $n$ is called a {\it quasi-bireflection} if its trace
has the form:
$$Tr_A(g,t) :=  \sum_{i=0}^\infty Tr (g|_{A_i}) t^i = \frac{1}{(1-t)^{n-2} q(t)} $$
where  $q(1) \neq 0$. As in the case of quasi-reflections, there are
``mystic quasi-bireflections" (quasi-bireflections that are not
bireflections of $A_1$) [Example \ref{xxex6.6}]. We prove the
following theorem in Section 4 by reducing to the commutative case.

\begin{theorem}
\label{xxthm0.5} Let $A$ be the quantum affine space
$k_{q}[x_1,\cdots,x_d]$ such that $q\neq \pm 1$. Let $G$ be a finite
subgroup of $\Aut(A)$. If $A^G$ satisfies (gci), then $G$ is
generated by quasi-bireflections.
\end{theorem}

Some partial results about $k_{-1}[x_1,\cdots,x_d]^G$ are given in
\cite{KKZ5}, including the following theorem, that can be regarded as 
a converse of the Kac-Watanabe-Gordeev Theorem.

\begin{theorem} \cite[Theorem 5.4]{KKZ5}
\label{xxthm0.6} Assume that $k$ is of characteristic
zero. Let $A=k_{-1}[x_1,\cdots,x_n]$ and $G$ be
a subgroup of the permutation group $S_n$ (acting naturally on $A$).
If $G$ is generated by quasi-bireflections, then  $A^G$ satisfies
(cci).
\end{theorem}


Further evidence that the definition of quasi-bireflection is a useful 
generalization comes from the study of invariants of noetherian graded 
down-up algebras,  a class of AS regular
algebras of global dimension 3. Let $Q_3$ be the finite subgroup
$GL_2(k)$ generated by
$$\begin{pmatrix} \epsilon &0\\0&\epsilon^{-1}\end{pmatrix}
\quad {\text{and}}\quad
\begin{pmatrix} -1 &0\\0&1\end{pmatrix}$$
where $\epsilon$ is a primitive $n$th root of unity for an odd integer
$n\geq 3$. Let $Q_4$ be the finite subgroup $GL_2(k)$ generated by
$$\begin{pmatrix} -\epsilon &0\\0&\epsilon^{-1}\end{pmatrix}
$$
where $\epsilon$ is a primitive $4n$th root of unity.


\begin{theorem}\cite[Theorem 0.3 and Table 4]{KKZ4}
\label{xxthm0.7} Assume that $k$ is of characteristic
zero. Let $A$ be a connected graded noetherian down-up algebra and
$G$ be a finite subgroup of $\Aut(A)$.
\begin{enumerate}
\item
If $A^G$ satisfies (gci), then $G$ is generated by
quasi-bireflections.
\item
The following conditions are equivalent:
\begin{enumerate}
\item[(i)]
$A^G$ satisfies (gci).
\item[(ii)]
$A^G$ is cyclotomic Gorenstein.
\item[(iii)]
$G$ is a finite subgroup of $GL_2(k)$ such that (iiia) $\det g=1$ or $-1$
for each $g\in G$ and (iiib) $G$ is not conjugate to $Q_3$ and $Q_4$
defined as above.
\end{enumerate}
\end{enumerate}
\end{theorem}

Theorem \ref{xxthm0.7} suggests that the ``cyclotomic Gorenstein''
property is closely related to our notions of complete
intersection for noncommutative algebras of the form $A^G$ where $A$
is a noetherian AS regular algebra and $G\subset \Aut(A)$. For the
generic 3-dimensional Sklyanin algebra $A:=A(a,b,c)$, we also show that
$A^G$ is cyclotomic Gorenstein if and only if $G$ is generated by
quasi-bireflections [Theorem \ref{xxthm5.5}]. 

Section 6 contains examples and questions for further study.





\section{Noncommutative versions of complete intersection}
\label{xxsec1}

Throughout let $k$ be a commutative base field of characteristic zero. Vector spaces, algebras 
and morphisms are over $k$.

In this section we propose several different, but closely related,
definitions of a noncommutative complete intersection graded
$k$-algebra.  We begin by recalling some definitions that will be 
used in our work. The Hilbert series of an ${\mathbb N}$-graded 
$A$-module $M=\bigoplus_{i=0}^\infty M_i$ is defined to be the 
formal power series in $t$
$$H_M(t)=\sum_{i=0}^{\infty} (\dim M_i) t^i.$$

\begin{definition}
\label{xxdef1.1}
Let $A=\bigoplus_{i=0}^\infty A_i$ be a connected graded algebra.
\begin{enumerate}
\item
$A$ has {\it exponential growth} if $\overline{\lim}_i
(\dim A_i)^{1/i}>1$.
\item
$A$ has {\it sub-exponential growth} if $\overline{\lim}_i
(\dim A_i)^{1/i}\leq 1$.
\item
The {\it Gelfand-Kirillov dimension} (or {\it GK-dimension},
for short) of $A$ is
$$\GKdim A=\overline{\lim}_n \log_n(\dim \bigoplus_{i=0}^n A_i).$$
\end{enumerate}
\end{definition}

We refer to \cite{KL} for the original definition of GK-dimension and
its basic properties. Definition \ref{xxdef1.1}(c) agrees with the
original definition given in \cite{KL} when $A$ is finitely generated,
but it may differ otherwise.

Next are the proposed definitions of a noncommutative complete 
intersection, beginning with the definition of AS regularity; 
a commutative AS regular algebra 
is a commutative polynomial ring, and hence it is reasonable to replace a 
commutative polynomial ring by a noetherian AS regular algebra 
in our notions of complete intersection.
By ``an element in a graded algebra" we will mean a homogeneous
element. A set of (homogeneous) elements
$\{\Omega_1,\cdots,\Omega_n\}$ in a graded algebra $A$ is called
{\it a regular sequence of normalizing elements} if, for each $i$, the
image of $\Omega_i$ in $A/(\Omega_1,\cdots,\Omega_{i-1})$ is regular
(i.e. a non-zero-divisor in $A/(\Omega_1,\cdots,\Omega_{i-1})$)
and normal in $A/(\Omega_1,\cdots,\Omega_{i-1})$
(i.e. if $\overline{\Omega_{i}}$ denotes the image of $\Omega_i$
in $\overline{A} = A/(\Omega_1,\cdots,\Omega_{i-1})$, then 
$\overline{\Omega_{i}}\; \overline{A} = \overline{A} \; \overline{\Omega_{i}}$).

\begin{definition}
\label{xxdef1.2} Let $A$ be a connected graded algebra. We call $A$
{\it Artin-Schelter Gorenstein} (or {\it AS Gorenstein} for short)
if the following conditions hold:
\begin{enumerate}
\item[(a)]
$A$ has injective dimension $d<\infty$ on the left and on the right,
\item[(b)]
$\Ext^i_A(_Ak,_AA)=\Ext^i_{A}(k_A,A_A)=0$ for all $i\neq d$, and
\item[(c)]
$\Ext^d_A(_Ak,_AA)\cong \Ext^d_{A}(k_A,A_A)\cong k(l)$ for some $l$
(where $l$ is called the {\it AS index}).
\end{enumerate}
If, in addition,
\begin{enumerate}
\item[(d)]
$A$ has finite global dimension, and
\item[(e)]
$A$ has finite Gelfand-Kirillov dimension,
\end{enumerate}
then $A$ is called {\it Artin-Schelter regular} (or {\it AS regular}
for short) of dimension $d$.
\end{definition}
\begin{definition}
\label{xxdef1.3}
Let $A$ be a connected graded finitely generated algebra.
\begin{enumerate}
\item
We say $A$ is a {\it classical complete intersection} (or {\it cci},
for short) if there is a connected graded noetherian AS regular
algebra $C$ and a regular sequence  of normalizing elements $\{\Omega_1,
\cdots,\Omega_n\}$ of positive degree such that $A\cong C/(\Omega_1,
\cdots,\Omega_n)$.
\item
The {\it cci number} of $A$ is defined to be
$$cci(A)=\min\{n \mid A\cong C/(\Omega_1,\cdots,\Omega_n) \;
{\text{as in part (a)}}\}.$$
\item
We say $A$ is a {\it hypersurface} if $cci(A)\leq 1$.
\end{enumerate}
\end{definition}

Let $E(A)$ denote the $\Ext$-algebra $\Ext^*_{A}(k,k)$ of $A$. It
is ${\mathbb Z}^2$-graded and can be viewed as a connected graded
algebra by using either the cohomological degree or the Adams grading 
(i.e. the grading inherited from $A$).

\begin{definition}
\label{xxdef1.4}
Let $A$ be a connected graded finitely generated algebra.
\begin{enumerate}
\item
We say $A$ is a {\it complete intersection of noetherian type}
(or {\it nci}, for short) if the $\Ext$-algebra $E(A)$ is a left
and right noetherian ring.
\item
When $A$ is a nci, the {\it nci number} of $A$ is defined to be
$$nci(A)=\Kdim E(A)_{E(A)}$$
where $\Kdim$ is the Krull dimension (of a right module).
\item
We say $A$ is an {\it n-hypersurface} if $nci(A)\leq 1$.
\end{enumerate}
\end{definition}

In the next definition, we consider $E(A)=\bigoplus_{i=0}^\infty
\Ext^i_{A}(k,k)$ as a connected graded algebra by using the cohomological
degree, so that the degree $i$ piece is $E_i=\Ext^i_A(k,k)$.

\begin{definition}
\label{xxdef1.5}
Let $A$ be a connected graded finitely generated algebra.
\begin{enumerate}
\item
We say $A$ is a {\it complete intersection of growth type} (or
{\it gci}, for short) if $E(A)$ has finite GK-dimension.
\item
The {\it gci number} of $A$ is defined to be
$$gci(A)=\GKdim (E(A)).$$
\item
We say $A$ is a {\it g-hypersurface} if $gci(A)\leq 1$.
\item We say
$A$ is a {\it weak complete intersection of growth type} (or {\it
wci}) if $E(A)$ has subexponential growth.
\end{enumerate}
\end{definition}

If $E(A)$ has finite GK-dimension then it has subexponential growth.
Hence a gci must be a wci. We recall the $\chi$-condition. Let $k$
also denote the graded $A$-bimodule $A/A_{\geq 1}$.

\begin{definition}\cite[Definition 3.2]{AZ}
\label{xxdef1.6} Let $A$ be a noetherian, connected graded algebra.
We say that the {\it $\chi$-condition} holds for $A$ if
$\Ext^j_A(k,M)$ is finite dimensional over $k$ for all finitely
generated graded left (and right) $A$-modules $M$ and all $j \geq
0$.
\end{definition}

The following lemma is easy.

\begin{lemma}
\label{xxlem1.7} Let $A$ be connected graded. Then the following are
equivalent.
\begin{enumerate}
\item
$A$ has finite global dimension and every term in the minimal
free resolution of the trivial module $k$ is finitely generated.
\item
$\Ext^i_A(k,k)$ is finite dimensional for all $i$ and
$\Ext^i_A(k,k)=0$ for all $i\gg 0$.
\item
$nci(A)=0$.
\item
$gci(A)=0$.
\end{enumerate}
If, further, $A$ is noetherian and satisfies
the $\chi$-condition, then the following are also
equivalent to the above.
\begin{enumerate}
\item[(e)]
$A$ is AS regular.
\item[(f)]
$cci(A)=0$.
\end{enumerate}
\end{lemma}

\begin{proof}
(a) $\Longleftrightarrow$ (e): This follows from \cite[Theorem
0.3]{Z1}. The rest is straightforward.
\end{proof}

In the commutative case, the different versions of complete intersection
are equivalent.

\begin{lemma}
\label{xxlem1.8}
Let $A$ be a connected graded
finitely generated commutative algebra. Then the following are
equivalent.
\begin{enumerate}
\item
$A\cong k[x_1, \dots, x_d]/(\Omega_1, \cdots, \Omega_n)$, where
$\{\Omega_1, \dots, \Omega_n\}$ is a regular sequence of homogeneous
elements in $k[x_1, \dots, x_d]$ with $\deg x_i>0$.
\item
$A$ is a cci.
\item
$A$ is a nci.
\item
$A$ is a gci.
\end{enumerate}
\end{lemma}

\begin{proof} By \cite{BH, FHT, FT, Gu, Ta}, it is well-known that
(a), (c) and (d) are equivalent
for local commutative rings.
The current graded version follows from \cite[Theorem IV.7]{FT} and
\cite[Theorem C]{FHT}.

(a) $\Longrightarrow$ (b): This is obvious.

(b) $\Longrightarrow$ (d): This follows from Theorem \ref{xxthm1.11}
which is to be proved later.
\end{proof}

In the noncommutative case, \ref{xxlem1.8}(a) and
\ref{xxlem1.8}(b,c,d) are clearly not equivalent. Example
\ref{xxex6.2} gives an algebra that is a gci, but is not a nci or a
cci.  Example \ref{xxex6.3} gives an algebra that is both a gci and
a nci, but is not a cci.

Every commutative complete intersection ring
$A:=k[x_1,\cdots,x_d]/(\Omega_1,\cdots,\Omega_n)$ is Gorenstein and
its Hilbert series is of the form
$$H_A(t)=\frac{\prod_{j=1}^n (1-t^{\deg \Omega_j})}
{\prod_{i=1}^d (1-t^{\deg x_i})}.$$
We use a similar condition to define a related notion.

\begin{definition}
\label{xxdef1.9}

Let $A$ be a connected graded finitely generated
algebra.
\begin{enumerate}
\item
We say a rational function $p(t)/q(t)$, where $p(t)$ and $q(t)$ are
coprime integral polynomials,  is {\it cyclotomic} if all the roots
of $p(t)$ and $q(t)$ are roots of unity.
\item
We say $A$ is {\it cyclotomic} if its Hilbert series $H_A(t)$ is of
the form $p(t)/q(t)$ and it is cyclotomic.
\item
The $cyc$ number of $A$ is defined to be
$$cyc(A)=\min\{m \mid H_A(t)=\frac{\prod_{i=1}^m
(1-t^{a_i})}{\prod_{j=1}^n (1-t^{b_j})}\}.$$
\item
We say $A$ is {\it cyclotomic Gorenstein} if the following
conditions hold
\begin{enumerate}
\item[(i)]
$A$ is AS Gorenstein;
\item[(ii)]
$A$ is cyclotomic.
\end{enumerate}
\end{enumerate}
\end{definition}

Let $A$ be the commutative ring $k[x,y,z]/(xy)$ where $\deg x=\deg
y=1$ and $\deg z=2$. Then the Hilbert series of $A$ is $1/(1-t)^2$  
and hence $cyc(A)=0$.  Then 
$E \cong k\langle x,y,z \rangle/(x^2,y^2,z^2, xz+zx, yz+zy)$  so 
that $cci(A)=nci(A)=gci(A)=1$.
It also is easy to construct algebras $A$ such that $cyc(A)=0$ and
$cci(A)=nci(A)=gci(A)$ is any positive integer. The notion of
cyclotomic Gorenstein is weaker than the notion of complete intersection,
even in the commutative case, as seen in Example \ref{xxex6.4}, an
example given by
Stanley. Thus we have the following implications in the commutative case.
\begin{equation}\label{E1.9.1}\tag{E1.9.1}
{\text{cci}}\quad \Longleftrightarrow \quad {\text{gci}}\quad
\Longleftrightarrow \quad {\text{nci}}\quad \Longrightarrow \quad
{\text{cyclotomic Gorenstein}}.
\end{equation}

Next we begin to relate these conditions in the noncommutative
case.

\begin{lemma}
\label{xxlem1.10} Let $A$ be connected graded and finitely generated.
Suppose $E:=E(A)$ is noetherian. Then $nci(A)=1$ if and only if
$gci(A)=1$.
\end{lemma}

\begin{proof} If $gci(A)=\GKdim E=1$ (and as $E$ is noetherian),
$E$ is PI by a result of Small-Stafford-Warfield \cite{SSW}. Hence
$\Kdim E_E=\GKdim E=1$  by \cite[Corollary 10.16]{KL}. Conversely, assume that $\Kdim E_E=nci(A)=1$.
By Lemma \ref{xxlem1.7} it suffices to show that $\GKdim E \leq1$. Since $E$ is noetherian, we
only need to show that $\GKdim E/{\mathfrak p}\leq 1$ for every prime
ideal ${\mathfrak p}$ of $E$ by \cite[Proposition 5.7]{KL}. When $R:=E/({\mathfrak p})$ is prime
there is a homogeneous regular element $f$ in $R$. Since
$\Kdim R\leq \Kdim E=1$, $\Kdim R/Rf=0$. Thus $R/Rf$ is finite
dimensional. So $R$ has GK-dimension 1.
\end{proof}

It is unknown if $nci(A)=2$ is equivalent to $gci(A)=2$ when $E$
is noetherian.

We say a sequence of nonnegative numbers $\{b_n\}_{n=0}^\infty$ has
subexponential growth (or the formal power series $\sum_{n=0}^\infty
b_n t^n$ has subexponential growth) if $\overline{\lim}_i
(b_i)^{1/i}\leq 1$. By \cite[Lemma 1.1]{SZ}, $\{b_n\}_{n=0}^\infty$
has subexponential growth if and only if $\{c_n:=\sum_{i=0}^n
b_i\}_{n=0}^\infty$ has subexponential growth. Here is our first
main theorem.

\begin{theorem}
\label{xxthm1.11} Suppose $A$ is connected graded and finitely
generated. Let $\Omega$ be a regular normal  element in $A$ of
positive degree and let $B=A/(\Omega)$.
\begin{enumerate}
\item
$gci(A)\leq gci(B)\leq gci(A)+1$, or equivalently,
$$\GKdim
\Ext^*_A(k,k)\leq \GKdim \Ext^*_B(k,k)\leq \GKdim \Ext^*_A(k,k)+1.$$
\item
$A$ is a wci if and only if $B$ is.
\item
If $cci(B)=1$, then $B$ is both a gci and a nci, and
$gci(B)=nci(B)=1$.
\end{enumerate}
\end{theorem}

\begin{proof}
(a) Let $a_n=\dim \Tor^A_n(k,k)$ and $b_n=\dim \Tor^B_n(k,k)$. Since
$\Tor^A_n(k,k)^*\cong \Ext^n_A(k,k)$, we have have $a_n=\dim
\Ext_A^n(k,k)$ and $b_n=\dim \Ext_B^n(k,k)$.

Since $B$ is a factor ring of $A$, there is a graded version of the
change-of-rings spectral sequence given in \cite[Theorem 10.71]{Ro}
\begin{equation}
{^2 E_{pq}}:=\Tor^{B}_p(k,\Tor^{A}_q(B,k))\Longrightarrow_p
\Tor^{A}_n(k,k). \label{E1.11.1}\tag{E1.11.1}
\end{equation}
Since $B=A/(\Omega)$ where $\Omega$ is a regular element in $A$,
$\Tor_0^A(B,k)=k$, $\Tor_1^A(B,k)=k$ and $\Tor_i^A(B,k)=0$ for
$i>1$. Hence the $E^2$-page of the spectral sequence \eqref{E1.11.1}
has only two possibly non-zero rows; namely

$q=0:\quad \Tor^{B}_p(k,k)$ for $p=0,1,2,\cdots$, and

\noindent

$q=1:\quad \Tor^{B}_p(k, k)$ for $p=0,1,2,\cdots$.

\noindent Since \eqref{E1.11.1} converges, we have
$\Tor^B_0(k,k)=\Tor^A_0(k,k)=k$ and a long exact sequence
$$\begin{aligned}
\cdots \cdots &\longrightarrow
\Tor^B_4(k,k)\longrightarrow \Tor^B_2(k,k)\longrightarrow \\
\longrightarrow \Tor^A_3(k,k)&\longrightarrow
\Tor^B_3(k,k)\longrightarrow \Tor^B_1(k,k)\longrightarrow \\
\longrightarrow \Tor^A_2(k,k)&\longrightarrow \Tor^B_2(k,k)
\longrightarrow \Tor^B_0(k,k)\longrightarrow \\
\longrightarrow \Tor^A_1(k,k)&\longrightarrow \Tor^B_1(k,k)
\longrightarrow 0. \notag
\end{aligned}
$$
As a part of the above long exact sequence, we have, for each $n$,
the following two exact sequences
\begin{equation}
\label{E1.11.2}\tag{E1.11.2} \Tor^B_{n-1}(k,k)\longrightarrow
\Tor^A_n(k,k) \longrightarrow \Tor^B_{n}(k,k),
\end{equation}
and
\begin{equation}
\label{E1.11.3}\tag{E1.11.3}\Tor^A_{n+2}(k,k)\longrightarrow
\Tor^B_{n+2}(k,k)\longrightarrow \Tor^B_n(k,k) \longrightarrow
\Tor^A_{n+1}(k,k).
\end{equation}
Then \eqref{E1.11.2} implies that
\begin{equation}
\label{E1.11.4}\tag{E1.11.4} a_n\leq b_n+b_{n-1}
\end{equation}
for all $n$, and \eqref{E1.11.3} implies that
\begin{equation}
\label{E1.11.5}\tag{E1.11.5} |b_{n+2}-b_{n}|\leq a_{n+2}+a_{n+1}  
\end{equation}
for all $n$. Using Definition \ref{xxdef1.1}(c), \eqref{E1.11.4}
implies that
$$\GKdim \Ext^*_A(k,k) \leq \GKdim \Ext^*_B(k,k),$$
and \eqref{E1.11.5} implies that $b_{n}\leq \sum_{i=0}^n a_i$
for all $n$, whence
$$\GKdim \Ext^*_B(k,k) \leq \GKdim \Ext^*_A(k,k)+1.$$

(b)   If $b_n$ has subexponential growth, then \eqref{E1.11.4}
implies that $a_n$ has subexponential growth. On the other hand, if
$a_n$ has subexponential growth, by \cite[Lemma 1.1]{SZ},
$c_n:=\sum_{i=0}^n a_i$ has subexponential growth. Now
\eqref{E1.11.5} implies that $b_n$ has subexponential growth. Thus
the assertion in (b) follows.

(c)
We prove this case assuming only that $A$ has finite global dimension
(we do not need to use the fact that $A$ is AS regular).

Let $E^n(A)=\Ext^n_A(k,k)$ and $E^n(B)=\Ext^n_B(k,k)$ for every $n$.
There is a spectral sequence for $\Ext$-groups, similar to
\eqref{E1.11.1}, which gives rise to a long exact sequence, see
\cite[Theorem 6.3 and the discussion afterwards]{CS},
\begin{equation}
\label{E1.11.6}\tag{E1.11.6} E^{n-1}(A) \xrightarrow{\gamma}
E^{n-2}(B)\xrightarrow{d_2} E^n(B)\xrightarrow{\phi^*} E^n(A)  .
\end{equation}
When set $n=p$, this is the exact sequence \cite[(6.1)]{CS}. Then
\cite[Lemma 6.1 and Theorem 6.3]{CS} show that \eqref{E1.11.6} can
be interpreted as an exact triangle of right $E(B)$-modules. There
is another sequence that can be interpreted as an exact triangle of
left $E(B)$-modules, see \cite[(6.2)]{CS} and the discussion after
\cite[Theorem 6.3]{CS}. For our purpose, it is enough to use that
$d_2:E(B) \to E(B)$ is a right $E(B)$-module homomorphism, which is
special case of \cite[Lemma 6.1]{CS}. Therefore $d_2$ is a left
multiplication $l_x$ by some element $x\in E^2(B)$. Since $A$ has
finite global dimension $E^n(A)=0$ for all $n\gg 0$. By the exact
sequence \eqref{E1.11.6} and the fact that $E^n(A)=0$ for all $n\gg
0$, we have that $E(B)/xE(B)$ is finite dimensional. Thus $E(B)$ is
a finitely generated left $k[x]$-module. Therefore $E(B)$ is left
noetherian of GK-dimension $\leq 1$. By symmetry, $E(B)$ is right
noetherian, too. Therefore $E(B)$ is left and right noetherian, and
$B$ is a nci.

By part (a) $gci(B)\leq 1$. By Lemma \ref{xxlem1.8}, $gci(B)\neq 0$.
By Lemma \ref{xxlem1.10}, $gci(A)=nci(A)=1$, and consequently $A$ is
both a n-hypersurface and a g-hypersurface.
\end{proof}

Under the hypotheses of Lemma \ref{xxlem1.8} we have
\begin{equation}\label{E1.11.7}\tag{E1.11.7}
cci(A)=0\quad \Longleftrightarrow \quad gci(A)=0 \quad
\Longleftrightarrow \quad nci(A)=0.
\end{equation}
By Lemma \ref{xxlem1.10} and Theorem \ref{xxthm1.11}(c), if $A$ is a nci,
then
\begin{equation}\label{E1.11.8}\tag{E1.11.8}
cci(A)=1\quad \Longrightarrow \quad nci(A)=1 \quad
\Longleftrightarrow \quad gci(A)=1.
\end{equation}

It is unknown if $gci(A)=1$ implies that $nci(A)=1$ (without
assuming $A$ is a nci a priori) and if $gci(A)=1$ implies that
$cci(A)=1$. It would be nice to have examples to answer these basic
questions.

Here is our second main result; part (a) is Theorem
\ref{xxthm0.1} and parts (b) and (c) are Theorem \ref{xxthm0.2}.

\begin{theorem}
\label{xxthm1.12} Let $A$ be a connected graded finitely generated
algebra.
\begin{enumerate}
\item
If $A$ is a cci, then $A$ is a gci, and $gci(A)\leq cci(A)$.
\item
If $A$ is a gci or a nci, then $A$ is a wci.
\item
Suppose that $A$ is a noetherian wci and that the Hilbert series
$H_A(t)$ is a rational function $p(t)/q(t)$ for some coprime
integral polynomials $p(t),q(t)\in {\mathbb Z}[t]$ with
$p(0)=q(0)=1$. Then $A$ is cyclotomic.
\end{enumerate}
\end{theorem}

\begin{proof}

(a) Write $A=C/(\Omega_1,\cdots,\Omega_n)$ for a noetherian AS
regular algebra $C$ and a regular sequence of normalizing elements
$\{\Omega_1, \cdots, \Omega_n\}$ of positive degree. The assertion
follows from Theorem \ref{xxthm1.11}(a) and the induction on $n$.

(b) Let $E=E(A)$. If $A$ is a gci, then $\GKdim E<\infty$. This
implies that $E$ has subexponential growth. Hence $A$ is a wci.

If $A$ is a nci, then $E$ is noetherian.  By \cite[Theorem
0.1]{SZ}, $E$ has subexponential growth. Therefore $A$ is a wci.

(c) Consider a minimal free resolution of the trivial module
\begin{equation}
\label{E1.12.1}\tag{E1.12.1} \cdots \to P^{i}\to \cdots P^1\to
P^0\to k\to 0
\end{equation}
with $P^i=\bigoplus_{s=1}^{n_i} A(-d^i_s)$ for all $i\geq 0$.
In particular, $P^0=A$. The Hilbert series of $P^i$ is
$$H_{P^i}(t)=\sum_{s=1}^{n_i} t^{d^i_s} H_A(t)=
(\sum_{s=1}^{n_i} t^{d^i_s})H_A(t).$$ Using the additive property of
the $k$-dimension, the exact sequence \eqref{E1.12.1} implies that
$$1=\sum_{i=0}^{\infty} (-1)^i H_{P^i}(t)
=\sum_{i=0}^{\infty} (-1)^i (\sum_{s=1}^{n_i} t^{d^i_s})H_A(t),$$
which implies that
$$H_A(t)=\frac{1}{
\sum_{i=0}^{\infty} (-1)^i (\sum_{s=1}^{n_i}
t^{d^i_s})}=:\frac{1}{Q(t)}.$$ Since $H_A(t)=p(t)/q(t)$,
$Q(t)=q(t)/p(t)$. Since \eqref{E1.12.1} is a minimal resolution of
$k$, we obtain that $\Ext^i_A(k,k)=\bigoplus_{s=1}^{n_i} k(d^i_s)$
for all $i$. Clearly the $\Ext$-algebra $E:=\Ext^*_A(k,k)$ is
${\mathbb Z}\times {\mathbb Z}$-graded and the every nonzero element
in $k(d^i_s)$ has degree $(-d^i_s, i)$, where the first grading is
the Adams grading coming from the grading of $A$ and the second
grading is the cohomological grading. Since \eqref{E1.12.1} is
minimal, $d^i_s\geq i$ for all $s$. Reassigning degree by
$\deg(k(d^i_s)) =(d^i_s, i)$, $E$ becomes connected and ${\mathbb
N}\times {\mathbb N}$-graded. This regrading is equivalent to
letting $\Ext^i_A(k,k)=\bigoplus_{s=1}^{n_i} k(-d^i_s)$. In this
case the ${\mathbb Z}\times {\mathbb Z}$-graded Hilbert series of
$E$ now is
$$H_E(t,u)=\sum_{i=0}^{\infty} (\sum_{s=1}^{n_i} t^{d_s^i})u^i.$$
By definition of $Q(t)$, we have $Q(t)=H_E(t,-1)$. By hypotheses,
$A$ is a wci, so $E$ has subexponential growth. Hence the Hilbert
series $H_E(1, u)$ has subexponential growth. Write
$H_E(1,u)=\sum_{i\geq 0} e_i u^i$, where $e_i=\dim \Ext^i_A(k,k)$,
and let $F_n=\sum_{i=0}^n e_i$. By \cite[Lemma 1.1(1)]{SZ},
$\{F_n\}_n$ has subexponential growth. Write $E(t,1)=\sum_{n\geq
0}f_n t^n$, where $f_n$ is the number of $d^i_s$ that are equal to $n$.
Since each $d_s^i\geq i$ for all $s=1,\cdots,n_i$, $\sum_{i=0}^n
f_n\leq F_n$ for all $n$. Since $\{F_n\}_n$ has subexponential
growth, so does $\{f_n\}_n$. Since the absolute value of each
coefficient of the power series $H_E(t,-1)$ is no more than (the
absolute value of) each coefficient of the power series $H_E(t,1)$,
$H_E(t,-1)$ has subexponential growth. As noted before
$Q(t)=H_E(t,-1)$, and we conclude that the coefficients of $Q(t)$ have
subexponential growth. We have seen that $Q(t)=q(t)/p(t)$ and write
$p(t)= \prod_{i=1}^d (1-r_i t)$. By \cite[Lemma 2.1]{SZ},
$|r_i|\leq 1$. Since $p(t)$ has integral coefficients, $|r_i|=1$ for
all $i$. The proof of \cite[Corollary 2.2]{SZ} shows that each
$r_i$ is a  root of unity.

Since $H_A(t)=p(t)/q(t)$,  by \cite[Corollary 2.2]{SZ}, each root
of $q(t)$ is a root of unity. Therefore $A$ is cyclotomic.
\end{proof}

By Theorem \ref{xxthm1.12} we have the implications below.

\begin{equation}\label{E1.12.2}\tag{E1.12.2}\end{equation}
$$\begin{CD}
cci @. nci\\
@V {\text{Theorem \ref{xxthm1.12}(a)}} VV
@VV{\text{Theorem \ref{xxthm1.12}(b)}} V \\
gci @>{\text{Theorem \ref{xxthm1.12}(b)}}>> wci\\
@. @V{\text{assume $H_A(t)$ is rational}}V
{\text{Theorem \ref{xxthm1.12}(c)}}V \\
@. cyclotomic
\end{CD}$$

\section{Higher Veronese subrings are not cyclotomic}
\label{xxsec2}

In many examples, one proves that an algebra $A$ is not a complete
intersection of any type by showing that $A$ is not cyclotomic, see
the diagram \eqref{E1.12.2}. In this short section we show that most
Veronese subrings of quantum polynomial rings are not cyclotomic.
First we will use a very nice result of Brenti-Welker \cite[Theorem
1.1]{BW}. Recall from \cite{BW} that for $a,b,c\in {\mathbb N}$ the
partition number is defined by
$$C(a,b,c)=\vert \{(n_1,\cdots, n_b)\in {\mathbb N}^b\;|\;
\sum_{i=1}^b n_i=c, 0\leq n_i\leq a \;\forall \; i\}\vert.$$ For
example, $C(1,3,1)=3$.

\begin{lemma}
\label{xxlem2.1} \cite[Theorem 1.1]{BW} Let $(a_n)_{n\geq 0}$ be a
sequence of complex numbers such that for some $s,d\geq 0$ its
generating series $f(t):=\sum_{n\geq 0} a_n t^n$ satisfies
$$f(t)=\frac{h_0+\cdots h_s t^s}{(1-t)^d}.$$
Set $f^{(r)}(t)=\sum_{n\geq 0}a_{rn} t^n$, for any integer $r\geq
2$. Then we have
$$f^{(r)}(t)=
\frac{h_0^{(r)}+\cdots + h_m^{(r)} t^m}{(1-t)^d},$$ where
$m:=\max\{s,d\}$ and
$$h_i^{(r)}
=\sum_{j=0}^s C(r-1,d,ir-j)h_j,$$
for $i=0,\cdots,m$.
\end{lemma}

We will apply this lemma to the case when $f(t)=H_A(t)$ for a
connected graded algebra $A$ and $f^{(r)}(t) =H_{A^{(r)}}(t)$ where
$A^{(r)}$ is the $r$th Veronese subring of $A$ (i.e. the subring of 
elements of $A$ with degree a multiple of $r$). The following lemma
is easy to see.

\begin{lemma}
\label{xxlem2.2} Suppose $d\geq 2$.
\begin{enumerate}
\item
$C(a,d,0)=1$ for all $a\geq 0$.
\item
$C(r-1,2,r)=r-1$ for $r \geq 1$.
\item
If $d\geq 3$, then $C(r-1,d,r)\geq d$ for $r \geq2$, and 
the inequality is strict unless $r=2, d=3$. 
\end{enumerate}
\end{lemma}

\begin{proposition}
\label{xxpro2.3} Let $A$ be a connected graded algebra with Hilbert
series
$$H_A(t)=\frac{1+h_1 t\cdots+ h_s t^s}{(1-t)^d}$$
where $h_i\geq 0$ for all $i$. Then the $r$th Veronese subring
$A^{(r)}$ is not cyclotomic if one of the following conditions
holds.
\begin{enumerate}
\item
$r\geq 3$ and $H_A(t)=(1-t)^{-2}$.
\item
$r$ satisfies the inequality $C(r-1, d,r)>\max\{s,d\}$.
\item
$r\geq 2$ and $H_A(t)=(1-t)^{-d}$ and $d\geq 3$.
\end{enumerate}
\end{proposition}

\begin{proof} (a) In this case it is easy to see that
$H_{A^{(r)}}(t)=\frac{1+(r-1)t}{(1-t)^2}$. Hence $A^{(r)}$ is not
cyclotomic when $r\geq 3$.

(b) Let $m=\max\{s,d\}$. Then $m<C(r-1, d,r)$ by the hypothesis. By
Lemma \ref{xxlem2.1}, $h_0^{(r)}=1$ and
$$h_1^{(r)}=C(r-1,d,r)+\sum_{j=1}^s C(r-1,d,r-j)h_j
\geq C(r-1,d,r)> m.$$ If $1+h_1^{\langle
r\rangle}t+ \cdots+ h_m^{(r)} t^m = \prod_{i=1}^m (1+r_it)$, 
then $\sum_{i=1}^m  r_i = h_1^{(r)} > m$, so 
$h_0^{(r)}+\cdots+  h_m^{(r)} t^m$ has a root with absolute value
strictly greater  than 1. Hence $A^{(r)}$ is not cyclotomic.

(c) In this case $s=0$. The assertion follows from 
Lemma \ref{xxlem2.2}(c), noting the case $r=2, d=3$ gives the Hilbert series
$(1+3t)/(1-t)^3$, which is not cyclotomic. 
\end{proof}

Quantum polynomial rings are noetherian AS regular domains 
whose Hilbert series are of the form
$(1-t)^{-d}$. The following corollary can be used to state 
precisely when the $r$th Veronese subalgebra
of a quantum polynomial ring is a complete intersection.

\begin{corollary}
\label{xxcor2.4} Let $A$ be a connected graded  algebra.
\begin{enumerate}
\item
Suppose $A=k[t]$ where $\deg t=1$. For every $r\geq 2$,
$A^{(r)}\cong k[x]$ where $\deg x=r$. So $A^{(r)}$ is AS regular, and
consequently, $A^{(r)}$ is cyclotomic and a classical complete
intersection.
\item
Suppose $A$ is noetherian of global dimension 2 with Hilbert series
$(1-t)^{-2}$.
\begin{enumerate}
\item[(i)]
$A^{(2)}$ is a classical complete intersection (and hence
cyclotomic).
\item[(ii)]
For every $r\geq 3$, $A^{(r)}$ is not cyclotomic. Consequently,
$A^{(r)}$ is not a complete intersection of any type.
\end{enumerate}
\item Suppose the Hilbert series of $A$
is $(1-t)^{-d}$. If $d\geq 3$ and  $r\geq 2$, then $A^{(r)}$ is
not cyclotomic. Consequently, $A^{(r)}$ is not a complete
intersection of any type.
\end{enumerate}
\end{corollary}

\begin{proof} (a) This is obvious.

(bi) Quantum polynomial rings of dimension 2 are isomorphic to 
either $k_q[x,y] := k\langle x,y\rangle /(yx-qxy)$ or 
$k_J[x,y] := k\langle x,y \rangle/(xy-yx-x^2)$, and $A^{(2)}$ 
is the ring of invariants under the diagonal map sending 
$x \mapsto -x$ and $y \mapsto -y$.  It follows from 
\cite[Theorem 0.1]{KKZ4} that each of these rings of 
invariants is a hypersurface in an AS regular ring of dimension 3.

(bii) This follows from Proposition \ref{xxpro2.3}(a).

(c) This follows from Proposition \ref{xxpro2.3}(c). 
\end{proof}

Another special case is when $A$ is AS regular of global dimension
three and generated by two elements  of degree 1.

\begin{lemma}
\label{xxlem2.5} Suppose $H_A(t)=\frac{1}{(1-t)^2(1-t^2)}$.
If $r\geq 3$, then $A^{(r)}$ is not cyclotomic. Consequently,
$A^{(r)}$ is not a complete intersection of any type.
\end{lemma}

\begin{proof} 

We compute the Hilbert series of $A^{(r)} = \sum_{k}
a_{kr}t^{kr}$, when $\sum_{j} a_j t^j$ is the expansion of
$\frac{1}{(1-t)^2(1-t^2)}$. When $r$ is even the Hilbert series of
$A^{(r)}$ is
$$\frac{(r^2-4r+4)t^{2r} + (r^2+4r-8)t^r+4}{4(1-t^r)^3}.$$
For the numerator to be symmetric it is necessary for $r$ to be a
positive integer with $r^2-4r+4 = \pm 4$ or $0$, which happens only
when $r=2$ or $r=4$.  The Hilbert series for $r=4$ is
$$ \frac{1+6t^4 + t^8}{(1-t^4)^3},$$
which is not cyclotomic. When $r$ is odd the Hilbert series of
$A^{(r)}$ is
$$\frac{b_0 + b_1t^{r} + b_2t^{2r} + b_3 t^{3r} + b_4t^{4r} +
b_5t^{5r} + b_6t^{6r} + b_7t^{7r} + b_8t^{8r} + b_9t^{9r}}{4(1-t^{2r})^5} $$
where
$$b_0=4, \;\;\;\; b_1 = 3 + 4r + r^2, \;\;\;\;b_2 = -16+8r+ 4r^2,$$
$$ b_3 = -12 -8r+4r^2,\;\;\;\; b_4 = 24-24r-4r^2, \;\;\;\;
b_5 = 18-10r^2,\;\;\;\;b_6= -16+24r-4r^2,$$
$$  b_7 = -12 + 8 r + 4r^2,\;\;\;\; b_8 = 4-8r + 4r^2,\;\;\;\;
b_9 =3-4r+r^2.$$
For the numerator to be symmetric it is necessary
that $r$ is an odd positive integer $\ge 3$  with $3-4r+r^2 = \pm 4$
or $0$, which could happen only when $r=3$.  The Hilbert series for
$r= 3$ is
$$\frac{ 1 + 6t^3 + 11 t^6  - 21 t^{12} - 18t^{15} + 5 t^{18}
+ 12t^{21} + 4t^{24}}{(1-t^{6})^5},$$
which is not cyclotomic.
\end{proof}

\begin{remark}
\label{xxrem2.6}
~\
\begin{enumerate}
\item When $A$ is a noetherian AS regular ring of global
dimension three that is generated by two elements of degree 1,
then $H_A(t)=\frac{1}{(1-t)^2(1-t^2)}$ and $A^{(2)}$ is a cci by 
\cite[Proposition 1.3]{VdB}.
\item If $A$ is noetherian and AS Gorenstein of dimension $d$
then by \cite[Theorem 3.6]{JZ} $A^{(r)}$ is AS Gorenstein if and only
if $r$ divides $\ell$ (where $\ell$ is the AS index).
\end{enumerate}
\end{remark}

\section{Invariant theory of AS regular algebras and quasi-bireflections}
\label{xxsec3}

In this section we connect the study of complete
intersections with noncommutative invariant theory. First we review
some definitions.


\begin{definition}
\label{xxdef3.1} Let $A$ be a noetherian algebra.
\begin{enumerate}
\item
Given any $A$-module $M$, the $j$-number of $M$ is defined by
$$j(M)=\min\{i \mid \Ext^i(M,A)\neq 0\}\in {\mathbb N}\cup
\{\infty\}.$$
\item
$A$ is called {\it Auslander Gorenstein} if
\begin{enumerate}
\item
$A$ has finite left and right injective dimension;
\item
for every finitely generated left $A$-module $M$ and for every right
$A$-submodule $N\subset \Ext^i_A(M, A)$, $j(N)\geq i$,
\item
the above condition holds when left and right are switched.
\end{enumerate}
\item
$A$ is called {\it Auslander regular} if $A$ is Auslander Gorenstein
and has finite global dimension.
\item $A$ is called {\it
Cohen-Macaulay} if, for every finitely generated (left or right)
$A$-module $M$, $\GKdim (M)+j(M)=\GKdim A<\infty$.
\end{enumerate}
\end{definition}

In a couple of places we also use dualizing complexes, as well as
the Auslander and
Cohen-Macaulay properties of dualizing complexes, all of which are
quite technical. We refer to the paper \cite{YZ} for these
definitions. For a graded algebra, the existence of an Auslander
Cohen-Macaulay dualizing complex was proved for many classes of
noetherian algebras. It is conjectured that every
noetherian AS regular algebra is Auslander regular.

The notion of the {\em homological determinant of} $g \in G$, for 
$G$ a group of automorphisms of an AS-Gorenstein algebra $A$, 
was  defined by J{\o}rgensen and Zhang in \cite[Definition 2.3]{JZ}, 
using the local cohomology of $A$.  It is a particular group 
homomorphism from $\Aut{A}\rightarrow  k^*$.  We refer to \cite{JZ} 
for the definition and properties of the homological determinant.  
For the examples in this paper the homological determinant can be 
computed using Lemma \ref{xxlem3.6} below.

\begin{lemma}
\label{xxlem3.2} Let $E$ be the $\Ext$-algebra of $A$. Suppose $A$
is a nci and $E$ has an Auslander Cohen-Macaulay dualizing complex.
Then $nci(A)\leq gci(A)<\infty$.
\end{lemma}

\begin{proof} By using \cite[Corollary 2.18]{YZ} (for $d_0=0$
and $\Cdim=\GKdim$), we obtain $\Kdim E\leq \GKdim
E<\infty$, and the assertions follow.
\end{proof}

Combining Theorem \ref{xxthm1.12} with Lemma \ref{xxlem3.2}, under
some reasonable conditions, we have
\begin{equation}\label{E3.3.1}\tag{E3.3.1}
cci(A)\geq gci(A)\geq nci(A).
\end{equation}

Next we recall Molien's Theorem which will be used several times later. 
Let $A$ be a noetherian connected graded AS regular algebra and 
$g$ be a graded automorphism of $A$. The {\it trace} of $g$ acting 
on $A$ is defined to be the formal power series $Tr_A(g,t) :=  
\sum_{i=0}^\infty Tr (g|_{A_i}) t^i$, where $Tr (g|_{A_i})$ is 
the trace of the linear map $g$ restricted to the space of 
homogeneous elements of degree $i$.  Traces can be used to compute 
the Hilbert series of fixed rings.

\begin{theorem}[Molien's Theorem] 
{\rm{(}}\cite[Lemma 5.2]{JiZ}{\rm{)}} 
\label{xxthm3.3} 
Let $A$ be a connected graded $k$-algebra and let $G$ be a finite group of 
graded automorphisms of $A$.  Then the Hilbert series of the fixed subring is
$$H_{A^G}(t) \frac{1}{|G|} \sum_{g \in G} Tr_A(g,t).$$
\end{theorem}

Now we are ready to prove Theorem \ref{xxthm0.4}.

\begin{theorem}
\label{xxthm3.4} Assume that $k$ is of characteristic zero.
Let $R$ be a connected graded noetherian
Auslander regular algebra and $G$ be a finite subgroup of $\Aut(R)$. If
$R^G$ is a wci, then it is cyclotomic Gorenstein.  In particular, 
if $R^G$ is a (gci) or an (nci) then it is cyclotomic Gorenstein.
\end{theorem}

\begin{proof} By \cite[Corollaries 1.12 and 5.9]{Mo}, $R^G$ is 
noetherian (and $R$ is finite over $R^G$).
By Molien's Theorem (Theorem \ref{xxthm3.3}), 
the Hilbert series of $R^G$ is of the form
$\frac{1}{|G|} \sum_{g\in G} Tr_{R}(g,t)$ and each $Tr_{R}(g,t)$
is of the form $e_g(t)^{-1}$ by \cite[Theorem 2.3(4)]{JiZ}. By 
\cite[Lemma 1.6(e)]{KKZ1}, the zeros of $e_g(t)$ are all roots of
unity. So we can write $e_g(t)$ as $\frac{p_g(t)}{q_g(t)}$ 
where $p_g(0)=q_g(0)=1$ and $q_g(t)\in {\mathbb Z}[t]$. Let $q(t)$
be the common multiple of $q_g(t)$ for all $g$, then $H_{R^G}(t)
=\frac{p(t)}{q(t)}$ where $q(0)=1$, $q(t)\in {\mathbb Z}[t]$
and $p(t)\in {\mathbb C}[t]$. Since $p(t)=q(t) H_{R^G}(t)$,
$p(0)=1$ and $p(t)\in {\mathbb Z}[t]$. 

Let $\sigma$ be any graded algebra automorphism of $R$
and let $M$ be any finitely generated graded left $R$-module. 
Since $R$ is noetherian and AS regular, by \cite[Proposition 4.2]{JZ},
every $\sigma$-linear automorphism $f: M\to M$ is rational over 
$k$ in the sense of \cite[Definition 1.3]{JZ}. By 
\cite[Lemma 6.3]{JZ}, every module automorphism of a 
finitely generated graded left $R^G$-module is 
rational over $k$ in the sense of \cite[Definition 1.3]{JZ}. 
As a special case (by taking the automorphism to be the 
identity map of $R^G$), 
the Hilbert series of $R^G$ is a rational function. 
By Theorem \ref{xxthm1.12}(c), $R^G$ is
cyclotomic. It remains to show that $R^G$ is AS Gorenstein.

Write $H_{R^G}(t)=p(t)/q(t)$ with $p(t),q(t)\in {\mathbb Z}[t]$
and $p(0)=q(0)=1$. Further assume that $p(t)$ and $q(t)$ are co-prime. 
By the proof of Theorem \ref{xxthm1.12}(c), every root of $p(t)$ is a root 
of unity. Since $p(t)$ is an integral polynomial, we 
have $p(t^{-1})=\pm t^{d_1}
p(t)$ for some $d_1$. Since $R^G$ has finite GK-dimension, every
root of $q(t)$ is a root of unity (see the proof of \cite[Corollary
2.2]{SZ}). Hence $q(t^{-1})=\pm t^{d_2} q(t)$ for some $d_2$. It
follows from \cite[Theorem 6.4]{JZ} that $R^G$ is AS Gorenstein.
\end{proof}

When $A=R^G$ where $R$ is a connected graded noetherian Auslander
regular algebra and $G$ is a finite subgroup of $\Aut(R)$, we can
modify the diagram \eqref{E1.12.2} a little by changing
``cyclotomic" to ``cyclotomic Gorenstein".

This is a good place to mention a natural question.

\begin{question}
\label{xxque3.5} Let $A$ be a noetherian connected graded algebra
that is either a nci or a gci. Under what conditions must $A$ be
AS Gorenstein (or Gorenstein)?
\end{question}

Example \ref{xxex6.3} shows that  $A$ can be both a nci and a gci
yet still not be Gorenstein. However Theorem \ref{xxthm3.4} says
that for $A=R^G$, where $R$ is a noetherian Auslander regular
algebra, then $A$ must be AS Gorenstein.

Traces can also be used to compute the homological determinant of $g$.  

\begin{lemma}
\label{xxlem3.6} \cite[Lemma 2.6]{JZ}
Let $A$ be AS Gorenstein of injective dimension $d$ and let
$g\in \Aut(A)$. Suppose $g$ is $k$-rational in the sense of
\cite[Definitions 1.3]{JZ} {\rm{(}}e.g., if $A$ is AS regular, 
or $A$ is PI{\rm{)}}. Then the trace of $g$ is of the form
\begin{equation}
\label{E3.6.1}\tag{E3.6.1}
Tr_A(g,t)=(-1)^d (\hdet g)^{-1} t^l+{\text{lower degree terms}},
\end{equation}
when we express $Tr_A(g,t)$ as a Laurent series in $t^{-1}$.
\end{lemma}

If $A=k[x_1, \cdots, x_n]$ and $\deg(x_i) = 1$, then 
$$Tr_A(g,t) =\frac{1}{\prod(1-\lambda_it)}$$
where the product is taken over all $n$ eigenvalues $\lambda_i$ of $g$.  
Other basic properties of the trace of $g$ can be found in \cite{JiZ}. 

\begin{definition}
\label{xxdef3.7} Let $A$ be a noetherian connected graded AS
regular algebra of GK-dimension $n$. Let $g\in \Aut(A)$.
\begin{enumerate}
\item \cite[Definition 2.2]{KKZ1}
We call $g$  a {\it quasi-reflection} if its trace has the form:
$$Tr_A(g,t) = \frac{1}{(1-t)^{n-1} q(t)}$$
where $q(t)$ is an integral polynomial with $q(1) \neq
0$.
\item
We call $g$ a {\it quasi-bireflection} if its trace has the form:
$$Tr_A(g,t) = \frac{1}{(1-t)^{n-2} q(t)}$$
where $q(t)$ is an integral polynomial with $q(1) \neq 0$.
\end{enumerate}
As the classical case, a quasi-reflection may also viewed as a
quasi-bireflection for convenience.
\end{definition}

\begin{definition}
\label{xxdef3.8} Let $A$ be a noetherian AS regular algebra and let
$W$ be a subgroup of $\Aut(A)$.
\begin{enumerate}
\item
The $cci$-$W$-bound of $A$ is defined to be
$$cci(W/A)=\sup\{cci(A^G)\mid  {\text{$\forall$ finite subgroups
$G\subset W$ such that $cci(A^G)<\infty$}} \}.$$
\item
The $gci$-$W$-bound of $A$ is defined to be
$$gci(W/A)=\sup\{gci(A^G)\mid  {\text{$\forall$ finite subgroups
$G\subset W$ such that $gci(A^G)<\infty$}} \}.$$
\item
The $nci$-$W$-bound of $A$ is defined to be
$$nci(W/A)=\sup\{nci(A^G)\mid  {\text{$\forall$ finite subgroups
$G\subset W$ such that $nci(A^G)<\infty$}} \}.$$
\item
The $cyc$-$W$-bound of $A$ is defined to be
$$cyc(W/A)=\sup\{cyc(A^G)\mid  {\text{$\forall$ finite subgroups
$G\subset W$ such that $cyc(A^G)<\infty$}} \}.$$
\end{enumerate}
\end{definition}

\begin{example} 
\label{xxex3.9} Let $A$ be a noetherian AS regular algebra of global
dimension two that is generated in degree 1. Then by \cite[Theorem
0.1]{KKZ4},
$$cci(\Aut(A)/A)=gci(\Aut(A)/A)=nci(\Aut(A)/A)=cyc(\Aut(A)/A)=1.$$
\end{example}

\section{Graded twists}
\label{xxsec4}

In this section we will prove Theorem \ref{xxthm0.5} using the 
technique of graded twists developed in \cite{Z2}. We refer to 
results in that paper, noting that \cite{Z2} uses right modules, 
while we are using left modules.

Let $A$ be an
${\mathbb Z}^d$-graded algebra and $\Aut_{{\mathbb Z}^d}(A)$ be the
group of ${\mathbb Z}^d$-graded algebra automorphisms of $A$. Let
$\{\tau_1,\cdots,\tau_d\}$ be a set of commuting elements in
$\Aut_{{\mathbb Z}^d}(A)$. We define a twisting system of $A$ by
$\tau_{g}=\tau_1^{g_1}\cdots \tau_d^{g_d}$ for every
$g=(g_1,\cdots,g_d) \in {\mathbb Z}^d$. Then $\tau=\{\tau_g\mid g\in
{\mathbb Z}^d\}$ is a left (respectively, right) twisting system in
the sense of \cite[Definition 2.1]{Z2}. Let $A^\tau$ be the graded
twisted algebra by the twisting system $\tau$ \cite[Proposition and
Definition 2.3]{Z2}. Let $A-\Gr$ be the category of graded left
$A$-modules and for a graded left $A$-module $M$ let 
$M^\tau$ be the associated twisted graded left $A^\tau$-module 
\cite[Proposition and Definition 2.6]{Z2}. Then the
assignment $F: M\mapsto M^\tau$ defines an equivalence of categories
\cite[Theorem 3.1]{Z2}
$$A-\Gr \cong A^\tau-\Gr.$$
Let $(g)$ denote the degree $g$ shift of a graded module. Let $M$
and $N$ be finitely generated graded left modules over a left
noetherian ring $A$. Then we have
$$\Ext_A^i(M,N)\cong \bigoplus_{g\in {\mathbb
Z}^d}\Ext^i_{A-\Gr}(M,N(g)).$$

\begin{lemma}
\label{xxlem4.1} Let $A$ be an ${\mathbb N}^d$-graded finitely
generated algebra such that $A$ becomes a connected ${\mathbb
N}$-graded when taking the total degree. Suppose
$\tau=\{\tau_1,\cdots, \tau_d\}$ is a set of commuting elements in
$\Aut_{{\mathbb Z}^d}(A)$. Then $A$ is a gci if and only if $A^\tau$
is.
\end{lemma}

\begin{proof} Let $B=A^\tau$. It suffices to show that $\Ext^i_A(k,k)
\cong \Ext^i_B(k,k)$ as graded vector space. Let $F$ be the functor
sending $M$ to $M^\tau$. Since $A$ is connected graded by using the
total degree, one can check that the only simple graded module over
$A$ is $k(g)$ for $g\in {\mathbb Z}^d$. Using this fact we obtain
that $F(_Ak(g))\cong _{B} k(g)$ for all $g\in {\mathbb Z}^d$. Hence
$$\begin{aligned}
\Ext^i_A(k,k)&=\bigoplus_{g\in {\mathbb Z}^d}
\Ext^i_{A-\Gr}(k,k(g))\cong \bigoplus_{g\in {\mathbb Z}^d}
\Ext^i_{B-\Gr}(F(k),F(k(g)))\\
&\cong \bigoplus_{g\in {\mathbb Z}^d}
\Ext^i_{B-\Gr}(k,k(g))=\Ext^i_B(k,k).
\end{aligned}
$$
\end{proof}

\begin{lemma}
\label{xxlem4.2} Let $A$ and $\tau$ be as in Lemma \ref{xxlem4.1}.
Let $G$ be a subgroup of $\Aut_{{\mathbb Z}^d}(A)$ such that every
element $g\in G$ commutes with $\tau_i$ for all $i=1,\cdots,d$. Then
\begin{enumerate}
\item
$G$ is a subgroup of $\Aut_{{\mathbb Z}^d}(A^\tau)$ under the
identification $A=A^\tau$ (as a graded vector space).
\item
The restriction of $\tau$ on $A^G$ defines a twisting system of
$A^G$, which we still denote by $\tau$.
\item
$(A^\tau)^G$ is a graded twist of $A^G$ by $\tau$.
\end{enumerate}
\end{lemma}

\begin{proof} Let $B=A^\tau$ with multiplication $*$. By definition,
$A=B$ as graded vector spaces. The difference between $A$ and $B$ is
in their multiplication, see \cite[Proposition and Definition
2.3]{Z2}.

(a) For every $g\in G$, define $g_B: B\to B$ to be the graded vector
space automorphism $g$. We need to show that this is an algebra
homomorphism. For any $x,y\in B$,
$$\begin{aligned}
g_B(x* y)&=g(\tau^{\deg y}(x)y)=g(\tau^{\deg y}(x)) g(y)
\\
&=\tau^{\deg y}(g(x))g(y)=g(x)* g(y)=g_B(x)* g_B(y).
\end{aligned}$$
Therefore $g_B$ is a graded algebra automorphism of $B$, and the
assertion follows.

(b) Since $\tau_i$ commutes with $G$, each $\tau_i$ induces an
automorphism of $A^G$ by restriction. Since the restrictions of
$\tau_i$ on $A^G$ commute with each other, $\tau$ defines a twisting
system of $A^G$.

(c) As graded subspaces of $A$, we have $(A^G)^\tau=A^G=(A^\tau)^G$.
By using the twisting, we see that the multiplications of
$(A^G)^\tau$ and $(A^\tau)^G$ are the same. In fact,  $(A^G)^\tau=
(A^\tau)^G$ as subalgebras of $A^\tau$.
\end{proof}

\begin{proposition}
\label{xxpro4.3} Let $(A, \tau, G)$ be as in Lemma \ref{xxlem4.2}.
\begin{enumerate}
\item
$(A^\tau)^G$ is a gci if and only if $A^G$ is.
\item
For every $g\in G$, $Tr_A(g_A,t)=Tr_{A^\tau}(g_{A^\tau},t)$.
\item
$G\subset \Aut_{{\mathbb Z}^d}(A)\subset \Aut(A)$ is generated by
quasi-bireflections of $A$ if and only if $G\subset \Aut_{{\mathbb
Z}^d}(A^\tau)\subset \Aut(A^\tau)$ is generated by
quasi-bireflections of $A^\tau$.
\end{enumerate}
\end{proposition}

\begin{proof} (a) By Lemma \ref{xxlem4.2}(c), $(A^\tau)^G$ is
a graded twist of $A^G$. The assertion follows from Lemma
\ref{xxlem4.1}.

(b) This follows from the fact that $A=A^\tau$ as an ${\mathbb
N}$-graded vector space and $g_A=g_{A^{\tau}}$ as a graded vector
space automorphism.

(c) Follows from part (b) and the definition of quasi-bireflection.
\end{proof}

Now we consider the skew polynomial ring
$B=k_{p_{ij}}[x_1,\cdots,x_d]$ which is generated by
$\{x_1,\cdots,x_d\}$ and subject to the relations
$$x_jx_i=p_{ij}x_ix_j$$
for all $i<j$ where $\{p_{ij}\}_{1\leq i<j\leq w}$ is a set of
nonzero scalars.

This is a ${\mathbb Z}^d$-graded algebra when setting $\deg x_i=(0,\cdots,
0,1,0,\cdots ,0)=:e_i$, where $1$ is in the $i$th position. We say an
automorphism $g$ of $A$ is {\it diagonal} if $g(x_i)=a_i x_i$ for some
$a_i\in k^\times$. Then every ${\mathbb Z}^d$-graded algebra
automorphism of $A$ is diagonal. As a consequence, $\Aut_{{\mathbb
Z}^d}(A)=(k^\times)^{d}$.

\begin{lemma}
\label{xxlem4.4} Let $B$ be the skew polynomial ring $k_{p_{ij}}
[x_1,\cdots,x_d]$ and $G$ a finite subgroup of $\Aut_{{\mathbb
Z}^d}(A)$.
\begin{enumerate}
\item\cite[Lemma 3.2]{KKZ2}
 $B$ is a ${\mathbb Z}^d$-graded twist of the
commutative polynomial ring by $\tau= (\tau_1,\cdots,\tau_d)$ where
$\tau_i$ is defined by $\tau_i(x_s)=\begin{cases}p_{is}  x_s & i<s\\
x_s & i\geq s\end{cases}$ for all $s$.
\item
$G$ commutes with $\tau_i$.
\item
If $B^G$ is a gci, then $G$ is generated by quasi-bireflections.
\end{enumerate}
\end{lemma}

\begin{proof} (a,b) are checked directly.

(c) By part (a) $B=A^\tau$ where $A=k[x_1,\cdots,x_d]$. If $B^G$
is a gci, by Proposition \ref{xxpro4.3}(a), $A^G$ is a gci.
Since $A=k[x_1,\cdots,x_d]$, by the Kac-Watanabe-Gordeev Theorem (\cite{KW} or \cite{G1}),
$G\subset \Aut(A)$ is generated by bireflections (which are also
quasi-bireflections in the sense of Definition \ref{xxdef3.7}(b)).
By Proposition \ref{xxpro4.3}(c), $G\subset \Aut(B)$ is generated
by quasi-bireflections of $B$.
\end{proof}

The following lemma is well-known, see for example \cite{AC}.

\begin{lemma}
\label{xxlem4.5} Let $B$ be the skew polynomial ring $k_{p_{ij}}
[x_1,\cdots,x_d]$. Suppose that $p_{ij}p_{st}\neq 1$ for any $i<j$
and $s<t$. Then $\Aut(B)=(k^\times)^d$.
\end{lemma}

\begin{proof} First of all $(k^\times)^d=\Aut_{{\mathbb Z}^d}(B)
\subset \Aut(B)$. It remains to prove that every ${\mathbb
Z}$-graded algebra automorphism $g$ of $B$ is diagonal. Since
$p_{ij}\neq 1$ for all $i<j$, by \cite[Lemma 3.5(e)]{KKZ2} there is
a permutation $\sigma\in S_d$ and $c_s\in k^\times$ for each $s$
such that
$$g(x_i)=c_i x_{\sigma(i)}\qquad {\text{for every $\quad$ $i=1,\cdots,d$}}.$$
Consequently, $p_{ij}=p_{\sigma(i)\sigma(j)}$ for all $i,j$.
We need to show that $\sigma$ is the identity.  Suppose $\sigma$ is
not the identity.  
There are distinct integers $i_1,\cdots, i_k$ in
the set $\{1,\cdots, d\}$ for some $k>1$ such that
$$\sigma: i_1\to i_2\to \cdots \to i_{k-1}\to i_k\to i_1.$$
Hence
$$p_{i_1 i_2}=p_{i_2 i_3}=\cdots =p_{i_{k-1} i_k}=p_{i_k i_1}.$$ 
Let $i_t$ is the smallest among all $\{i_s\}_{s=1}^k$.
Then $p_{i_t i_{t+1}} p_{i_t i_{t-1}}=1$, a contradiction. Therefore 
$\sigma$ is the identity, completing the proof.
\end{proof}

Now we are ready to prove Theorem \ref{xxthm0.5}.

\begin{proof}[Proof of Theorem \ref{xxthm0.5}]
Let $C$ denote the algebra $k_{q}[x_1,\cdots,x_d]$. Then $C$ is a
special case of $k_{p_{ij}}[x_1,\cdots,x_d]$ by taking $p_{ij}=q$
for all $i<j$. Since $q^2\neq 1$ (or $q\neq \pm 1$), Lemma
\ref{xxlem4.5} says that $\Aut(C)=(k^\times)^d$. Let $G$ be a finite
subgroup of $\Aut(C)$. Then $G\subset (k^\times)^d$, and the assertion
follows from Lemma \ref{xxlem4.4}(c).
\end{proof}

\section{Some fixed subrings of the Sklyanin Algebra}
\label{xxsec5}

Let $A=A(a,b,c)$ be the 3-dimensional Sklyanin algebra, i.e. the
algebra generated by $x,y,z$ with relations
\[ax^2+byz + c zy =0\]
\[ay^2 + b zx + cxz = 0\]
\[az^2 + b xy + c yx = 0,\]
for $a,b,c \in k$.
By \cite{ATV} the algebra $A$ is noetherian and AS regular except
when $a^3=b^3=c^3$ or when two of the three parameters $a,b,c$ are
zero. We assume throughout that parameters $a,b,c$ are generic
enough so that none of $a, b, c$ is zero and the cubes $a^3, b^3,
c^3$ are not all equal so that $A(a,b,c)$ is AS regular, and we 
further assume that $A$ is not PI. Then by \cite[Example 10.1]{Sm},  
Koszul dual $A^{!}$ of $A$ is the three-dimensional algebra 
generated by $X, Y, Z$ with quadratic relations:
\[cYZ-bZY=0 \hspace{1 in} bX^2-aYZ=0\]
\[cZX-bXZ=0 \hspace{1 in} bY^2-aZX=0\]
\[cXY-bYX = 0 \hspace{1 in} bZ^2-aXY =0,\]
and cubic relations $$XY^2 = XZ^2=Y^2Z=Z^2X = 0.$$
Note that these cubic relations are consequence of the quadratic relations.
A vector space basis for $A^!$ is $\{1,X,Y,Z, XY, ZX, YZ, XYZ\}$.

For a graded automorphism $\sigma$, the action of $\sigma$ on the elements of 
degree 1, $\sigma|_{A_1}$ is given by a matrix, and the transpose of this 
matrix gives rise to an automorphism $\sigma^!$ of the dual algebra $A^!$ .
It follows that the  groups $\Aut(A)$ and $\Aut(A^!)$ are
anti-isomorphic (see discussion \cite[p. 267]{JZ}), and we can determine 
$\Aut(A)$ by computing $\Aut(A^!)$.

Let the matrix $(\alpha_{ij})$ represent the linear map $\sigma$
that takes
\begin{eqnarray*}
\sigma(X)& = &\alpha_{11}X + \alpha_{21}Y + \alpha_{31}Z,   \\
\sigma(Y) & = &\alpha_{12}X + \alpha_{22}Y + \alpha_{32}Z, \\
\sigma(Z) & = &\alpha_{13}X + \alpha_{23}Y + \alpha_{33}Z.
\end{eqnarray*}

\begin{lemma}
\label{xxlem5.1} The automorphisms of $A^!$ are of the following
forms:
\[g_1= \begin{pmatrix} \alpha \omega& 0& 0 \\ 0 & \alpha \omega^2& 0\\ 0 & 0 &\alpha
\end{pmatrix}, \;
g_2 = \begin{pmatrix}0 & \alpha \omega & 0\\ 0 & 0 & \alpha \omega^2
\\ \alpha & 0 & 0\end{pmatrix}, \;
g_3= \begin{pmatrix}0 & 0 & \alpha \omega\\ \alpha \omega^2 & 0 & 0\\
0& \alpha & 0\end{pmatrix},\]
 where $\alpha$ is arbitrary and
$\omega$ satisfies $\omega^3 =1$. The traces of these automorphisms are:
$$Tr_{A^!}(g_1, t) =1+ \alpha( 1 + \omega + \omega^2)t + 
\alpha^2( 1 + \omega + \omega^2)t^2 +\alpha^3t^3$$ 
$$Tr_{A^!}(g_i, t) =1+ \alpha^3t^3 \text{ for } i=2,3.$$ 
Hence the graded algebra automorphisms
of $A$ are transposes of these matrices, so of the same forms, and 
their homological determinants are all $\alpha^3$.
\end{lemma}

\begin{proof} Recall that we assume $a,b,c$ are generic,
and since $A$ is not PI we may assume $b \neq c$. Since $b \neq 0$, 
without loss of generality we may assume that $b=1$ and $c \neq 1$.  
Then the elements of $A^!$ of degree 2 form a 3-dimensional vector 
space with basis $XY, \;YZ,$ and $ZX$.  Let $g$ be a graded 
automorphism of $A$ and let $\sigma = g^!$ be the induced 
automorphism on $A^!$.  The automorphism  $\sigma$ must preserve 
the equations: $cXY-YX= 0, cZX-XZ = 0,$ and  $cYZ-ZY=0$ in $A^!$.  
Computing $\sigma(cXY-YX) = 0$ in terms of the basis elements 
$XY$, $YZ$, and $ZX$, and using the fact that $c \neq 1$, we see 
that the entries $\alpha_{i,j}$ of $\sigma$ must satisfy the equations:
\[ (1+c) \alpha_{1,2} \alpha_{2,1} + a \alpha_{3,1} \alpha_{3,2} =0 
\hspace*{.5in}   (\text{coefficient of } XY) \]
\[ (1+c) \alpha_{2,2} \alpha_{3,1} + a \alpha_{1,1} \alpha_{1,2} =0  
\hspace*{.5in} (\text{coefficient of } YZ) \]
\[ (1+c) \alpha_{3,2} \alpha_{1,1} + a \alpha_{2,1} \alpha_{2,2} =0  
\hspace*{.5in} (\text{coefficient of } ZX). \]
Similarly the equation $\sigma(cZX - XZ) = 0$ gives the equations:
\[ (1+c) \alpha_{2,3} \alpha_{1,1} + a \alpha_{3,1} \alpha_{3,3} =0 \]
\[ (1+c) \alpha_{3,3} \alpha_{2,1} + a \alpha_{1,1} \alpha_{1,3} =0 \]
\[ (1+c) \alpha_{1,3} \alpha_{3,1} + a \alpha_{2,1} \alpha_{2,3} =0, \]
and the equation $\sigma(cYZ-ZY) = 0$ gives the equations:
\[ (1+c) \alpha_{2,2} \alpha_{1,3} + a \alpha_{3,3} \alpha_{3,2} =0 \]
\[ (1+c) \alpha_{3,2} \alpha_{2,3} + a \alpha_{1,3} \alpha_{1,2} =0 \]
\[ (1+c) \alpha_{1,2} \alpha_{3,3} + a \alpha_{2,3} \alpha_{2,2} =0. \]
Using the equations above and equations resulting from applying 
$\sigma$ to the relations $X^2-aYZ=0, Y^2-aZX=0$ and $Z^2-aXY =0$, 
computations in Maple show that the only nonsingular matrices
satisfying these equations are matrices of the forms:
\[g_1= \begin{pmatrix} \gamma & 0& 0 \\ 0 & \beta& 0\\ 0 & 0 &\alpha
\end{pmatrix}, \;
g_2 = \begin{pmatrix}0 & \gamma & 0\\ 0 & 0 & \beta
\\ \alpha & 0 & 0\end{pmatrix}, \;
g_3= \begin{pmatrix}0 & 0 & \gamma \\ \beta & 0 & 0\\
0& \alpha & 0\end{pmatrix}\]
with $\alpha^3=\beta^3 = \gamma^3$.
Taking $\alpha \neq 0$ arbitrary, these final three relations show 
that $\sigma$ must be of one of the three forms indicated and it 
is easy to check that these linear maps are algebra automorphisms of $A$.

The traces of the $g_i$ acting as automorphisms of the 3-dimensional 
algebra $A^!$ can be computed from the definition of trace, and the 
traces of the induced automorphisms on $A$ can be computed using the 
formula \cite[Corollary 4.4]{JZ}
$$Tr_A(g, t) = \frac{1}{Tr_{A^!}(g^{!} ,-t)}.$$ 
Using the $Tr_A(g_i,t)$ and  Lemma \ref{xxlem3.6} we see that 
$\hdet_A(g_i) = \alpha^3$.
\end{proof}

If $a,b,c$ are not generic, there are more graded algebra
automorphisms of $A$. To specify $\omega$ and $\alpha$, we also use the notation
$g_i(\alpha,\omega)$ for the matrices (or the automorphisms) listed
in Lemma \ref{xxlem5.1}.  

By \cite[Corollary 6.3]{KKZ1}, $A$ has no quasi-reflection of finite
order. By \cite[Corollary 4.11]{KKZ3}, if $G$ is a finite subgroup
of $\Aut(A)$ and $A^G$ is Gorenstein, then $\hdet g=1$ for all $g\in
G$.  Hence for $g_i $ of the forms above, if  $g_i \in G$ with $A^G$ 
AS Gorenstein then we must have $\alpha^3=1$ by Lemma \ref{xxlem3.6}.
The following lemma follows this observation and direct computation.

\begin{lemma}
\label{xxlem5.2} Let $A=A(a,b,c)$ and $SL(A)$ be the subgroup of
$\Aut(A)$ generated by $g_1$, $g_2$ and $g_3$ with
$\alpha^3=\omega^3=1$.
\begin{enumerate}
\item
The order of $SL(A)$ is 27.
\item
$\hdet g=1$ for all $g\in SL(A)$.
\item
If $G$ is a subgroup of $\Aut(A)$ and every $g \in G$ has trivial homological
determinant, then $G$ is a subgroup of $SL(A)$.
\end{enumerate}
\end{lemma}

Now we consider only $g\in G$ with trivial homological determinant.
For $g$ to be a quasi-bireflection we need $t=1$ to be a root of
$Tr_A(g,t)$ of multiplicity 1, which occurs when $t=-1$ is a root of
$Tr_{A^!}(g^!,t)$ of multiplicity 1, and this happens in each case
when $\alpha^3 = 1$ and in the first case when, in addition, we have
$\omega \neq 1$ (i.e. we eliminate the case of a scalar matrix); then,
in each case, the automorphism $g$ is a quasi-bireflection with
$Tr_A(g,t) = 1/(1-t^3)$ and $\hdet g =1$ by Lemma \ref{xxlem3.6}.  
We summarize these facts in the following lemma.

\begin{lemma}
\label{xxlem5.3} Retaining the notation above we have:
\begin{enumerate}
\item
$g$ is a quasi-bireflection of $A$ if and only if $g$ is of the form
$g_1,g_2$ or $g_3$ with $\alpha^3=\omega^3=1$, and $\omega\neq 1$
when $g=g_1$. As a consequence, $g\in SL(A)$.
\item
If $g$ is a quasi-bireflection of $A$, then $Tr_A(g,t)=(1-t^3)^{-1}$.
\end{enumerate}
\end{lemma}

We can classify all subgroups of $SL(A)$.

\begin{lemma}
\label{xxlem5.4} The complete list of subgroups of $SL(A)$ is as follows:
\begin{enumerate}
\item
$\{1\}$.
\item
order 3 subgroups generated by single element, namely,
$$\langle g_1(\alpha,\omega)\rangle, \quad
\langle g_2(\alpha,\omega)\rangle, \quad {\text{ and}}\quad
\langle g_3(\alpha,\omega)\rangle$$
for any pair $(\alpha,\omega)$ with $\alpha^3=\omega^3=1$.
In the case of $g_1$, $(\alpha,\omega)\neq (1,1)$.
\item
order 9 subgroup $G_1:=\langle g_1(\alpha,\omega)\mid
\alpha^3=\omega^3=1\rangle$.
\item
order 9 subgroups
$$G_2:=\langle g_1(\alpha\neq 1,\omega=1),
g_2\rangle, \quad {\text{and}}
\quad G_3:=\langle g_1(\alpha\neq 1,\omega=1), g_3\rangle.$$
\item
the whole group $SL(A)$.
\end{enumerate}
\end{lemma}

\begin{proof} If the order of the subgroup is  $ \leq 3$, then
clearly we get cases (a) and (b). It is obvious that the subgroups $G_1,
G_2,G_3$ in parts (c,d) are of order $9$. Now assume that $G$ is a
subgroup of $SL(A)$ of order 9, which is not of the form in part
(c). So $G$ contains either $g_2$ or $g_3$. By symmetry, we assume
that $g_2\in G$. Since the order of every element in $G$ is either 1
or 3, $G\cong C_3\times C_3$. So $G$ is abelian, so there are two
elements of order $g_2$, say $a=g_2(\alpha_1, \omega_1)$ and
$b=g_2(\alpha_2,\omega_2)$. Then $ab^{-1}=g_1(\alpha, \omega)$.
Since the order of $G$ is 9, $\omega=1$. Thus we have the group
$G_2$ or $G_3$.
\end{proof}

\begin{theorem}
\label{xxthm5.5} Retaining the notation above, we have  $cyc(\Aut(A)/A)=1$,
and the following are equivalent for subgroups $\{1\}\neq G\subset SL(A)$.
\begin{enumerate}
\item
$G$ is not $\langle g_1(\alpha, 1)\rangle$.
\item
$G$ is generated by quasi-bireflections.
\item
$A^G$ is cyclotomic Gorenstein.
\end{enumerate}
\end{theorem}

\begin{proof}
Since $Tr_A(g_1(\alpha, 1),t)=(1-\alpha t)^3$,  $g_1(\alpha, 1)$ is
not a quasi-bireflection when $\alpha\neq 1$. Hence $G=\langle
g_1(\alpha, 1)\rangle$ is not generated by quasi-bireflections. The
fixed subring $A^G$ is the Veronese $A^{(3)}$. By Corollary \ref{xxcor2.4}(c),
$A^G$ is not cyclotomic.

It is straightforward to check that all other groups $G$ are
generated by quasi-bireflections. So it remains to show that $A^G$
is cyclotomic Gorenstein in each of these cases.  In each case we compute the Hilbert series of the fixed ring using Molien's Theorem (Theorem \ref{xxthm3.3}).

If $G$ has order 3 with generator $g$ where $g$ is either $g_2, g_3$
or $g_1$ with $\omega\neq 1$, then the fixed ring $A^g$ has Hilbert
series
\[\frac{1}{3(1-t)^3} + \frac{2}{3(1-t^3)} =
\frac{1 - t + t^2}{(1-t)^2(1-t^3)} =
\frac{1-t^6}{(1-t)(1-t^2)(1-t^3)^2}.\] So $A^G$ is cyclotomic
Gorenstein. If $G$ is the any group of order 9 in Lemma
\ref{xxlem5.4}, then the Hilbert series of $A^G$ is
$$\begin{aligned}
H_{A^G}(t)&=\frac{1}{9}\biggl( \frac{1}{(1-t)^3}+\frac{1}{(1-\omega
t)^3}+ \frac{1}{(1-\omega^2 t)^3}+\frac{6}{(1-t^3)}\biggr)\\
&=\frac{1+t^3+t^6}{(1-t^3)^3}=\frac{1-t^9}{(1-t^3)^4}.
\end{aligned}
$$
where $\omega$ is a primitive 3rd root of unity. So $A^G$ is
cyclotomic Gorenstein.

Finally, if $G=SL(A)$, then the Hilbert series of $A^G$ is
$$\begin{aligned}
H_{A^G}(t)&=\frac{1}{27} \left( \frac{1}{(1-t)^3}+
\frac{1}{(1-\omega t)^3}+ \frac{1}{(1-\omega^2 t)^3}+
\frac{24}{(1-t^3)} \right)\\
&= \frac{1-t^3+t^6}{(1-t^3)^3} = \frac{1+t^9}{(1-t^3)^2(1-t^6)} =
\frac{1-t^{18}}{(1-t^3)^2(1-t^6)(1-t^9)}.
\end{aligned}$$
So $A^G$ is cyclotomic Gorenstein.

Finally, by the above Hilbert series, we see that
$cyc(\Aut(A)/A)=1$.
\end{proof}

We make the following conjecture.

\begin{conjecture}
\label{xxcon5.6} Let $A=A(a,b,c)$ where $a,b,c$ are generic. Then
the following are equivalent for any finite subgroup $G\subset
\Aut(A)$.
\begin{enumerate}
\item
$A^G$ is a cci.
\item
$A^G$ is a gci.
\item
$A^G$ is a nci.
\item
$A^G$ is cyclotomic Gorenstein.
\item
$G$ is generated by quasi-bireflections.
\end{enumerate}
\end{conjecture}

\section{Examples and questions}
\label{xxsec6}

In this section we collect some examples and questions which
indicate the differences between the commutative and the noncommutative
situations. First we show that $cci$, $gci$ and $nci$ are different concepts.

\begin{lemma}
\label{xxlem6.1} Suppose $A$ is a finite dimensional algebra. If $A\cong
C/(\Omega_1,\cdots,\Omega_n)$ where $C$ has finite global dimension
and  $\{\Omega_1, \dots, \Omega_n\}$ is a regular sequence of normalizing
elements in $C$, then $C$ is noetherian, AS regular, Auslander
regular and Cohen-Macaulay, and $A$ is Frobenius.
\end{lemma}

\begin{proof} Since $A$ is noetherian, $C$ is noetherian by
\cite[Proposition 3.5 (a)]{Le}.
Furthermore $C$ has enough normal elements in the sense of \cite[p.
392]{Z1}. By \cite[Theorem 0.2]{Z1}, $C$ is AS regular,
Auslander regular and Cohen-Macaulay. By Rees's lemma
\cite[Proposition 3.4(b)]{Le}, $A$ is Gorenstein of injective
dimension 0. Hence $A$ is Frobenius.
\end{proof}

\begin{example}
\label{xxex6.2} Let $A=k\langle x,y\rangle/(x^2,xy,y^2)$. This is a
finite dimensional Koszul algebra with Hilbert series $1+2t +t^2$,
and $A$ is finitely generated and noetherian. The $\Ext$-algebra $E=E(A)$
is isomorphic to $k\langle x,y\rangle/(yx)$ with Hilbert series
$(1-t)^{-2}$. By definition, $gci(A)=\GKdim E=2$ and $A$ is a gci.
It is well-known that $E$ is not (left or right) noetherian.
 For example, if
$I_i:= Exy \oplus Exy^2 \oplus \cdots \oplus Exy^i$  and 
$J_i := xyE \oplus x^2yE \oplus \cdots\oplus x^iyE$, then 
$I_i$ (respectively, $J_i$) gives an infinite ascending chain 
of left (respectively, right) ideals
of $E$, so $E$ is not (left or right) noetherian and $\Kdim E=\infty$.
By definition, $A$ is not a nci and $nci(A)=\infty$. Since $A$ is
finite dimensional and not Frobenius ($A$ is local with 2 minimal 
right ideas $xA$ and $yxA$), $A$ is not a 
cci [Lemma \ref{xxlem6.1}] and $cci(A)=\infty$. Since
$H_A(t)=\frac{(1-t^2)^2}{(1-t)^2}$, $cyc(A)=2$.
\end{example}

\begin{example}
\label{xxex6.3} Let $R$ be the connected graded Koszul noetherian
algebra of global dimension four that is not AS regular given in
\cite[Theorem 1.1]{RS}. Its Hilbert series is $H_R(t)=(1-t)^{-4}$.
Let $A$ be the Koszul dual of $R$. Since $R$ is not AS regular, $A$
is not Frobenius \cite[Theorem 4.3 and Proposition 5.10]{Sm}. Hence $A$
is not a cci by Lemma \ref{xxlem6.1}. However, the $\Ext$-algebra of
$A$ is $R$, which is noetherian and has GK-dimension 4.
Consequently, $A$ is both a gci and nci. It is easy to see that
$gci(A)=cyc(A)=4$.
\end{example}

An example of Stanley shows that, even in the commutative case,
$cci$ and ``cyclotomic Gorenstein'' are different concepts.

\begin{example}\cite[Example 3.9]{St2}
\label{xxex6.4} Let $A$ be the connected commutative algebra
$$k[x_1,x_2,x_3,x_4,x_5,x_6,x_7]/I$$  with $\deg(x_i)=1$ where
$$I = (x_1x_5-x_2x_4, x_1x_6-x_3x_4, x_2x_6-x_3x_5, x_1^2x_4-x_5x_6x_7,
x_1^3-x_3x_5x_7).$$ This is a normal Gorenstein domain, but not a
complete intersection (so $gci(A)=\infty$). Its Hilbert series is
$$H_A(t) =\frac{(1+t)^3}{(1-t)^4}=\frac{(1-t^2)^3}{(1-t)^7}.$$
Hence it is cyclotomic Gorenstein and $cyc(A)=3$.
\end{example}

Many examples and Theorem \ref{xxthm1.12} indicate that being a cci
could be the strongest among all different versions of a
noncommutative complete intersection. In this direction, we need to
answer the following question.

\begin{question}
\label{xxque6.5} Suppose ${\text{char}}\; k=0$. If $A$ is a cci,
then is $A$ a nci?
\end{question}

By Theorem \ref{xxthm1.11}(c), this conjecture is true when
$cci(A)\leq 1$, and is unsolved for $cci(A)\geq 2$. If Question
\ref{xxque6.5} has a positive answer, then we have the following
diagram for invariant subrings under a finite group acting on an AS-regular algebra. 
\begin{equation}\label{E6.5.1}\tag{E6.5.1} \quad \end{equation}
$$\begin{CD}
\qquad\qquad cci \qquad\qquad @>\not\leftarrow {\text{Example
\ref{xxex6.3}}} >{\text{Conjecture \ref{xxque6.5}}}> nci\\
@V {\text{Theorem \ref{xxthm1.12}(a)}} V\not\uparrow {\text{Example
\ref{xxex6.3}}}V
@VV{\text{Theorem \ref{xxthm1.12}(b)}} V \\
\qquad\qquad gci \qquad\qquad @>{\text{Theorem \ref{xxthm1.12}(b)}}>> wci\\
@. @VV{\text{Theorem \ref{xxthm3.4}}}V \\
 @. cyclotomic \;\; Gorenstein\\
\end{CD}$$
where the implication from $wci\longrightarrow  cyclotomic \;\; Gorenstein$ is valid
only for fixed subrings of Auslander regular algebras [Theorem \ref{xxthm3.4}].

The next two
 examples concern the noncommutative version of a
bireflection. We would like to show that quasi-bireflection is a
good generalization in the noncommutative setting in order to answer
Question \ref{xxque0.3}.

\begin{example}
\label{xxex6.6} Let $B=k_{-1}[x,y,z]$ where all the variables
$(-1)$-skew commute. Let $V=B_1=kx+ky+kz$. Consider
$$g =\begin{pmatrix}0 & -1 & 0\\1& 0 &0\\0 & 0 & -1\end{pmatrix},$$
which has eigenvalues $-1, i ,-i$. Hence $g\mid_V$ is not a
classical bireflection of the vector space $V$.  By an easy computation,
$$Tr_A(g,t) = 1/((1+t)(1-t^2)) = Tr_A(g^2,t)= Tr_A(g^3,t)$$
and  $\hdet g=1$ 
by Lemma \ref{xxlem3.6}, and hence $g$ is a quasi-bireflection.  
A computation using Molien's Theorem (Theorem \ref{xxthm3.3}) shows
the Hilbert series of the fixed ring is
$$H_{A^\langle g \rangle}(t) = \frac{1-t^6}{(1-t^2)^3(1`-t^3)},$$
and
that $A^{g}$ is isomorphic (under an isomorphism that associates 
$x^2+y^2 \mapsto X$, $xy \mapsto Y$, $z^2 \mapsto Z$ and 
$x^2z-y^2z \mapsto W$) to
$$k[X,Y,Z,W]/(W^2-(X^2+4Y^2)Z),$$
 which is a
commutative complete intersection. To match our terminology for
quasi-reflections in \cite{KKZ1}, we might call $g$ a mystic
quasi-bireflection because $g$ is a quasi-bireflection but not a
classical bireflection of $V$.
\end{example}

\begin{example}
\label{xxex6.7} Let $B= k_{-1}[x,y,z,w]$ where all the variables
$(-1)$-skew commute. Let $V=B_1=kx+ky+kz+kw$. Let
$$g =\begin{pmatrix} 0 & 1 & 0&0\\1 & 0 & 0 &0\\
0 & 0 & 0 & 1\\ 0 &0 & 1& 0 \end{pmatrix}.$$ Then $g$ is a classical
bireflection of the vector space $V$ but $g$ is not a
quasi-bireflection, since its trace is $Tr(g,t) = 1/(1+t^2)^2$. The
fixed subring is Gorenstein because $\hdet g =1$ by Lemma \ref{xxlem3.6}. 
Using Molien's Thereom (Theorem \ref{xxthm3.3}) the Hilbert series
of the fixed ring is
$$H_{B^g}(t)=\frac{1-2t+4t^2-2t^3+t^4}{(1-t)^4(1+t^2)^2}$$
so that $B^g$ is Gorenstein, but not cyclotomic Gorenstein.
Consequently, $B^G$ is not a noncommutative complete intersection of
any type.
\end{example}

In \cite{WR} all groups $G$ so that $\mathbb{C}[x,y,z]^G$ is a
complete intersection are completely determined. In the
noncommutative case, we can ask:

\begin{question}
\label{xxque6.8} Let $A$ be a noetherian AS regular algebra of
global dimension three that is generated in degree 1. Determine all
finite subgroups $G\subset \Aut(A)$ such that $A^G$ is a
noncommutative complete intersection (of a certain type).
\end{question}

This question was answered for global dimension 2 in \cite[Theorem
0.1]{KKZ4}. The problem of determining which groups $G$ of graded
automorphisms of $A=k[x_1, \cdots, x_n]$ have the property that
$A^G$ is a hypersurface was solved in \cite{N1,N2} (see also
\cite[Theorem 7]{NW}). So we can ask

\begin{question}
\label{xxque6.9} Let $A$ be a noetherian AS regular algebra of
global dimension at most three that is generated in degree 1.
Determine all finite subgroups $G\subset \Aut(A)$ such that $A^G$ is
a hypersurface (of a certain type).
\end{question}

Another closely related question is

\begin{question}
\label{xxque6.10} If $nci(A)=1$, then is $A$ cyclotomic Gorenstein?
\end{question}

The complete intersections $A^G$ for $A=\mathbb{C}[x,y,z]$ and $G$  an
abelian subgroup of $GL(n, \mathbb{C})$ are computed in \cite{W3}.
We can ask, in the noncommutative case, what can be said when $G$ is
abelian?

In \cite{WR} complete intersections $\mathbb{C}[x,y,z]^G$
are considered; here $\mathbb{C}[x,y,z]^G$ is a complete
intersection if and only if the minimal number of algebra generators
of $\mathbb{C}[x,y,z]^G$ is $\leq 5$ \cite[Corollary p. 107]{NW}. A
complete intersection $A^G$ for $A=k[x_1, \cdots, x_n]$ always has
$\leq 2n -1$ generators (but for n=4 there is an example of $A^G$
not a complete intersection but the number of generators is $7$).
Are there noncommutative version(s) of these results?  What
nice properties distinguish noncommutative complete intersections
from other AS Gorenstein rings?

\begin{remark}
\label{xxrem6.11} To do noncommutative algebraic geometry and/or
representation theory (e.g. suppport varieties and so on) related to
noncommutative complete intersection rings sometimes it is
convenient to assume that the $\Ext$-algebra of $A$ has nice
properties. For example, one might want to assume that $A$ is a nci
and $E(A)$ has an Auslander Cohen-Macaulay rigid dualizing complex.
Together with these extra hypotheses, this should be a good
definition of noncommutative complete intersection. So we conclude
with the following question.
\end{remark}

\begin{question}
\label{xxque6.12} Let $A$ be a connected graded noetherian algebra
that is a nci. Suppose that the $\Ext$-algebra $E(A)$ has an
Auslander Cohen-Macaulay rigid dualizing complex. Then is $A$
Gorenstein?
\end{question}

\section*{Acknowledgments}
The authors thank Quanshui Wu for several useful discussions and
valuable comments and thank 
Brian Smith who assisted with some of the initial computations of
examples related to this work. Ellen
Kirkman was partially supported by the Simons Foundation (\#208314)
and James Zhang by the National Science Foundation (DMS 0855743 and
DMS-1402863).

\end{document}